\documentclass[11pt]{article}
\usepackage{t1enc}
\usepackage{epsf}
\usepackage{amsthm}
\usepackage{amsmath,amssymb}
\usepackage{graphicx}
\usepackage{xcolor}
\newcommand{\N}{\mbox{I\kern-2ptN}}
\newcommand{\R}{\mbox{I\kern-2ptR}}
\newcommand{\C}{\mbox{I\kern-6ptC}}

\newtheorem{teo}{Theorem}

\newtheorem{co}[teo]{Corollary}

\theoremstyle{definition}
\newtheorem{nota}[teo]{Remark}
\newtheorem{ejem}[teo]{Example}

\usepackage{amsfonts}

\usepackage[ansinew]{inputenc}

\begin{document}
\title{
The NIEP: Irreducible and Positive Realizations
\thanks{%
Partially supported by grant PID2022-138906NB-C21 funded by MCIU/AEI/10.13039/501100011033 and by ERDF "A way of making Europe", and the Spanish Research Council (Comisi\'on Interministerial de Ciencia y Tecnolog\'ia) under project PID2021-122501NB-I00.} \\
}
\author{ 
C. R. Johnson, C. Mariju\'an, M. Pisonero}  
\author{ 
C. R. Johnson $^{\rm a}$, C. Mariju\'an $^{\rm b}$, M. Pisonero $^{\rm b,}$
 \thanks{Corresponding author.\newline 
{\it E-mail addresses:} crjohn@wm.edu (C. R. Johnson), cmarijuan@uva.es 
(C. Mariju\'an), mpisonero@uva.es (M. Pisonero).}}    
 
\date
{\small {\it 
$^{\rm a}$Dept. Mathematics, College of William and Mary, Williamsburg, Virginia, 23187 USA\\
$^{\rm b}$Dpto. Matem\'atica Aplicada, Universidad de Valladolid/IMUVa, Spain
}}
\maketitle

\begin{abstract} Our focus is upon {\it irreducible} nonnegative $n$-by-$n$ matrix realizations of nonnegatively realizable spectra or, equivalently, characteristic polynomials. After giving some general background, we make some useful new observations and show the existence of irreducible nonnegative realizations in some general cases. Then, we focus on $n<5$, where the NIEP is solved. Finally, we focus on the trace 0 case and, using graph theoretic methods, characterize nonnegative irreducible realizability among realizable polynomials. The closely related problem of positive realizations, for trace positive spectra, is also discussed.

\end{abstract}

\

\section{Introduction}\label{S1}

By an $n$-spectrum $\sigma$, we mean a list of $n$ complex numbers, including possible allowed repeats. If implicit, the $n$ is omitted. We will think of these as sets (though repeats are allowed) and abuse notation and terminology in natural ways, {\it e.g.}, $\sigma ^2$ is the $n$-list of element-wise squares from $\sigma$.

In the Nonnegative Inverse Eigenvalue Problem (NIEP), $\sigma$ is said to be {\it realizable}, $\mathcal{R}$, if there is an $n$-by-$n$ entry-wise nonnegative matrix $A$, whose list of eigenvalues (with possible repeats) is $\sigma$. There are well-known necessary conditions for realizability of $\sigma$:

$\mbox{Tr} (\sigma)$, the sum of all elements of $\sigma$, must be nonnegative (trace condition)

$\sigma$ must be self-conjugate (reality)

\noindent and

the largest absolute value of an element of $\sigma$, $\rho(\sigma)$, the so-called Perron root, must occur in $\sigma$ (Perron condition).

There may be ties or repeats that may require additional structure in realizations. However, these conditions are far from sufficient.

Another approach to  the NIEP is to view a nonnegative matrix as the adjacency matrix of a weighted digraph and focus the attention on the coefficients of its characteristic polynomial $P(x)=x^n+k_1x^{n-1}+ \cdots+k_n$, see \cite{TAAMP}. The polynomial $P(x)$ is said to be {\it realizable}, $\mathcal{R}$, if there is a weighted digraph (equivalently an entry-wise nonnegative matrix) whose characteristic polynomial is $P(x)$.

The general NIEP is one of the most profound and difficult problems in all of mathematics. See \cite{HJ1} for general background and \cite{JMPP} for a recent survey, as well as other  references \cite{LoLo,LM2,Sp} for further background. However, the difficulty, means that parts of, or variations upon the full NIEP,  can be very interesting or worthwhile, and there are many \cite{Pe1,PeMi,Fi,Mi1}.

Our interest here lay in identifying, among realizable spectra/polynomials, those that are realizable by an irreducible matrix/a strongly connected weighted digraph. Recall that an $n$-by-$n$ matrix is {\it irreducible} if it is not similar by a permutation matrix to a nontrivially block triangular matrix. A digraph is strongly connected if every two vertices are joined by a path.  As $\mathcal{R}$ identifies the realizable spectra/polynomials, we use $\mathcal{IR}$ to indicate the irreducibly realizable ones. Since the $\mathcal{IR}$ spectra/polynomials are closely related to the spectra/polynomials realizable by an entry-wise positive matrix, we denote the latter by $\mathcal{R}_{>0}$. Of course $\mathcal{R}_{>0} \subseteq \mathcal{IR} \subseteq \mathcal{R}$, and each containment is strict.

It is also well known \cite{F}, that if $\sigma$ is the spectrum of an irreducible nonnegative matrix $A$, that $\rho(\sigma)$ has multiplicity 1 in $\sigma$ and that $\rho(\sigma)$ enjoys an entry-wise positive left and right eigenvector associated with $A$. There could still be ties for the spectral radius $\rho$ among other eigenvalues in $\sigma$. The occurrence of such ties is significant, however. When none occur, we call $\sigma$ {\it spectrally simple} (SS, for short), { \it i.e.},  $\rho(\sigma)$ is the only eiegenvalue on the spectral circle. It will be convenient to use some special notation: for $t$ real we denote the spectrum, in which $\sigma$ is modified only by adding $t$ to $\rho(\sigma)$, by $\sigma_{t}$; $t$ may be a positive or negative real number.

When $\sigma$ is not SS (but enjoys the trace, reality and Perron conditions), we say that it has {\it Frobenius structure} if there are $h>1$ eigenvalues (in $\sigma$) on the circle of radius $\rho(\sigma)$, and these are evenly spaced about that circle. That means that they are multiples of the $h$-th roots of unity by $\rho(\sigma)$. Then any remaining nonzero eigenvalues must be inside this first circle. If we choose one of the next largest absolute values say $\lambda$, with $|\lambda|=a$, there must exist $h$ of them with absolute value $a$ equally spaced around the circle of radius $a$. If none of these is $a$, their conjugates must also appear, by reality. Then continue looking inside the second circle for the next largest absolute value, which may or may not also be $a$ and continue as before. On each interior, nontrivial circle there must be a self-conjugate set of eigenvalues of cardinality a multiple of $h$. Any zero eigenvalues are ``wild cards'', and there could be any number of them. The characteristic polynomial of $\sigma$ with $\rho(\sigma)=\rho$ is of the form
$$
x^p\left(x^h-\rho^h\right)\left(x^h-(z_2\rho)^h\right) \cdots \left(x^h-(z_m\rho)^h\right)
$$
with $z_j\in \C^*$ and $|z_j|<1$ for $j=2,\dots,m$. When $\sigma$ is the spectrum of a nonnegative matrix, and $\sigma$ is not SS, then wether $\sigma$ has Frobenius structure is very important. According to \cite{F,G,HJ1} we have

\begin{teo}\label{Teo1}
If $\sigma$ $(P(x))$ is $\mathcal{R}$  and not SS, then if  $\sigma$ $(P(x))$ is $\mathcal{IR}$ it has Frobenius structure.
\end{teo}

\begin{nota}
We suspect a converse, that does not seem to be in the literature: ``If $\sigma$ is  $\mathcal{R}$ and not SS, then $\sigma$ is $\mathcal{IR}$ if and only if $\sigma$ has Frobenius structure''.
\end{nota}

 Note that $\sigma$ may be not SS and not enjoy Frobenius structure and still be realizable. It may be that a subset of the eigenvalues may be realizable and not fit the Frobenius structure of the others.
 
 \begin{ejem}\label{ejem3}
The spectrum $\sigma=\{2,1,-2\}$, or equivalently the polynomial $P(x)=x^3-x^2-4x+4=(x^2-4)(x-1)$, is not SS, does not have Frobenius structure so is not $\mathcal{IR}$ but is $\mathcal{R}$.
 \end{ejem}

 \begin{ejem}
An example that nontrivially fits the Frobenius structure is
$$
\sigma=\left\{1,-1,\frac{1}{2} \pm\frac{i}{2},-\frac{1}{2} \pm\frac{i}{2}\right\} \quad \mbox{or} \quad P(x)=x^6-x^4+\frac{1}{4}x^2-\frac{1}{4}=(x^2-1)\left(x^2+\frac{i}{2}\right)\left(x^2-\frac{i}{2}\right).
$$
They are $\mathcal{IR}$ and have the realizations 
$$
\left[\begin{array}{cccccc}
0&1&0&0&0&0\\
\frac{1}{2}&0&1&0&0&0\\
0&0&0&1&0&0\\
0&0&0&0&1&0\\
0&0&0&0&0&1\\
\frac{1}{4}&0&0&0&\frac{1}{2}&0\\
\end{array}
\right]
\quad \mbox{or}\quad 
\left[\begin{array}{cc}
0 & B\\
I&0
\end{array}
\right]
$$
where 
$$
B=\frac{1}{3}\left[\begin{array}{ccc}
1& 1+\frac{\sqrt{3}}{2} & 1-\frac{\sqrt{3}}{2}\\
1-\frac{\sqrt{3}}{2}&1 & 1+\frac{\sqrt{3}}{2}\\
1+\frac{\sqrt{3}}{2} &1-\frac{\sqrt{3}}{2} &1
\end{array}
\right]
$$
has spectrum $\left\{1, \pm \frac{i}{2}\right\}$. Note that Theorem \ref{t39} can not be applied to $P(x)$, but is $\mathcal{IR}$ and so $\mathcal{R}$.
 \end{ejem}

After considerable thought and work, we have come to a general conjecture for the $\mathcal{IR}$ problem. In some ways this seems natural, almost a converse to the classical structure, but far from trivial.

\

\noindent {\bf Conjecture:}  If $\sigma$ $(P(x))$ is $\mathcal{R}$ and SS, then $\sigma$ $(P(x))$ is $\mathcal{IR}$.

\

We prove this conjecture for $n$-spectra with $n<4$ and give considerable evidence for $n=4$. Since, for $n\geq 5$, the $\mathcal{R}$ spectra are not all known (except for many special cases), it is difficult to address the $\mathcal{IR}$ problem for $n\geq 5$. However, we also give several general cases and some interesting examples.

This paper is organized as follows. In the next section (\ref{S2}) we give some necessary background that supports our work. We then proceed from some general results to particular dimensions. In Section \ref{S3} we discuss which realizable spectra  are necessarily $\mathcal{IR}$. Though simple, the main observation is quite powerful. In Section \ref{S4}, we discuss the question of when adding 0's to $\sigma$ can preserve $\mathcal{IR}$. We then discuss (Section \ref{S5}) some real spectra that are $\mathcal{IR}$ in any dimension followed by, the close but not exact, relation between $\mathcal{IR}$ and its special case $\mathcal{R}_{>0}$. In Section \ref{S7} we consider negation invariant spectra, which are closely allied with Frobenius structure. Then, in sections \ref{S8}, \ref{S9} and \ref{S10}, we analyze the cases $n=2,3,4$. Already in case $n=3$, not all $\mathcal{R}$ spectra are $\mathcal{IR}$. In Section \ref{S11}, we discuss $\mathcal{IR}$ spectra (polynomials) with trace 0.

\section{Background and Observations}\label{S2}

A handy tool for adjusting an eigenvalue of a matrix, without putting it in any canonical form, was observed by Brauer \cite{Brauer}. It is more general than we mention, but we give the case that we use and that is most natural for nonnegative matrix analysis. Its use is often based on the orthogonality of right and left eigenvectors associated with different eigenvalues, the ``principle of biorthogonality'' \cite{HJ1}.

\begin{teo}\label{Brauer}
If $A$ is a nonnegative matrix with spectrum $\sigma$, Perron root $\rho$ and associated (nonnegative) normalized right eigenvector $\mathbf{x}$, and $\mathbf{y}$ is an arbitrary vector, then $A+\epsilon \,\mathbf{xy}^T$ has spectrum $\{\rho +\epsilon \,\mathbf{y}^T\mathbf{x}\}\cup (\sigma-\{\rho\})$. If $\mathbf{x}$ is positive and $\mathbf{y}$ and $\epsilon$ are chosen positive, then  $A+\epsilon\, \mathbf{xy}^T$ is positive with Perron root $\rho +\epsilon \,\mathbf{y}^T\mathbf{x}$.
\end{teo}

\begin{co}
If a nonnegative matrix $A$ has a positive right (or left) eigenvector associated with its Perron root, then there is a positive matrix arbitrarily close to $A$ that has the same eigenvalues as $A$, except that the Perron root has increased by an arbitrarily small amount.
\end{co}

It is well known \cite{J} that an irreducible nonnegative matrix $A$ of order $n$ is similar to an irreducible matrix $B$ with constant row sums $\rho(A)$: since $A$ has a positive (right) eigenvector $(x_1,\dots,x_n)$ associated with its spectral radius, then $B=D^{-1}AD$, where $D=diag(x_1,\dots,x_n)$, is a nonnegative matrix similar to $A$ and with constant row sums $\rho(A)$.

For a reducible nonnegative matrix $A$ of order $n$, only co-spectrality to a nonnegative matrix  with constant row sums can be assured. In this case, the {\it Frobenius Normal Form}, arrived at by a suitable permutation matrix $P$, can reduce $A$ to a block triangular form

\begin{equation*}
P^{-1}AP=
\begin{bmatrix}
A_{11}     &          &          &            &              &         &  \\
0          & A_{22}   &          &            &              &         &  \\
\vdots     & \ddots   & \ddots   &            &              &         & \\
0          & \cdots   & 0        & A_{kk}     &              &         &  \\
A_{k+1,1}  & \cdots   & \cdots   & A_{k+1,k}  &  A_{k+1,k+1} &         &  \\
\vdots     &          &          & \vdots     &              & \ddots  &  \\
A_{k+r,1}  & \cdots   &   \cdots       & A_{k+r,k}  & \cdots       &  \cdots       & A_{k+r,k+r}
\end{bmatrix}
\end{equation*}
in which each block $A_{ii}$ is square and is either  irreducible  or a $1$-by-$1$ null matrix, and $\begin{bmatrix}
A_{k+i,1} \, \dots \, A_{k+i,k+i-1}
\end{bmatrix}$ is nonzero, for $i=1,\dots ,r$. The blocks $A_{ii}$ are called the {\it irreducible components} or {\it classes} of $A$. The blocks $A_{ii}$, $i=1,\dots,k$, are called the {\it final components} or {\it classes} of $A$ and the blocks $A_{ii}$ with $\rho(A_{ii})=\rho(A)$ are called the {\it basic components} or {\it classes} of $A$. If $A$ has a simple Perron root and its basic component is a final component, then $A$ has a positive eigenvector associated with $\rho(A)$, and $A$ is similar to a matrix $B$ with constant row sums, see \cite{JMPS,JMPS1}. When the Perron root is not simple, we have the following characterization:

\begin{teo}{\rm (\cite[Theorem 6.5.2]{G})}
Let $A$ be a reducible nonnegative matrix with spectral radius $\rho(A)$. Then there exists a positive eigenvector $(x_1,\dots,x_n)$ associated with $\rho(A)$ if and only if all the final components of $A$ are basic components and the non final components of $A$ have strictly smaller spectral radius. (i.e., the basic components are exactly the final components)
\end{teo}

In this case, as in the irreducible case, $B=D^{-1}AD$, where $D=diag(x_1,\dots,x_n)$, is a nonnegative matrix similar to $A$ and with constant row sums $\rho(A)$.

By a {\it weighted digraph} $G$, or simply {\it digraph}, we mean a triplet 
$(V,E,w)$ where $V$ is a nonempty finite set, $E\subset V\times V$ and 
$w \colon E\to \R ^+ $ is a positive real map on $E$. The elements in  $V$ 
 and  $E$  are called {\it vertices} and {\it arcs} respectively; the values of the map $w$
are called {\it weights}. The {\it adjacency matrix} of a weighted digraph $(V,E,w)$ with 
$V=\{v_1,\dots ,v_n\}$ is the matrix $A=(a_{ij})_{i,j=1}^{n}$ where $a_{ij}=w(v_i,v_j)$ if
 $(v_i,v_j)\in E$ and $a_{ij}=0$ otherwise. An arc of the form $(v_j,v_j)$ is called {\it loop}. Note that $A$ is a nonnegative matrix.

A sequence of different vertices  $v_1,v_2,\dots , v_r$,  $r\geq 2$, such that  
$(v_i,v_{i+1}) \in E$ for $i= 1,\dots ,r-1$  is called a {\it  path}  of {\it length}
$r-1$ joining  $v_1$ with $v_r$. A {\it cycle} of {\it length} $r$ or an $r$-{\it cycle} is 
a sequence of vertices $v_1,v_2,\dots , v_r,v_1$ where $v_1,v_2,\dots , v_r$ is a path and
$(v_r,v_{1}) \in E$.  The cycles of length 1 or 1-cycles are the loops. A {\it linear digraph} is a collection of disjoint cycles.

\begin{teo}\label{teocof}
{\rm \bf (Coefficient Theorem for weighted digraphs)} Let $G$ 
be a weighted digraph, $A$ its adjacency matrix and 
$P_G(x)=P_A(x)=\det (xI-A)=x^{n}+k_{1}x^{n-1}+k_{2}x^{n-2}+\dots+k_{n}$. Then, for each 
$1\leq i\leq n$,
\begin{equation}
\label{1.1}
k_i=\sum_{L\in \cal {L} \it _i} (-1)^{p(L)}\Pi (L)
\end{equation}
\noindent where $\cal {L}\it _i$ is the set of all linear subdigraphs $L$ of $G$ with 
exactly $i$ vertices; $p(L)$ denotes the number of cycles of $L$;
 $\Pi (L)$ denotes the product of the weights of all arcs belonging to $L$. {\rm (See \cite{CDS})}
 \end{teo}
 
In the last section we will use the following result:
 
\begin{teo} {\rm (\cite[Theorem 39]{TAAMP})} \label{t39}
Let $p$ and $n$ be integers, such that $2\leq p\leq n\leq 2p+1$. Let $P(x)=x^n+k_px^{n-p}+\cdots +k_{n-1}x+k_n$. Then 
the following statements are equivalent:
\vspace*{-.2cm}
\begin{itemize}
\item [{\it i)}] $P(x)$ is realizable; 
\vspace*{-.2cm}
\item [{\it ii)}] the coefficients of $P(x)$ satisfy:
\vspace*{-.3cm}
\begin{itemize}
\item [{\it a)}] $k_p,\ldots ,k_{2p-1}\leq 0$;
\item [{\it b)}] $k_{2p}\leq\frac{k_p^2}{4}$;
\item [{\it and c)}] $k_{2p+1}\leq\left\{
\begin{array}{ll}
	k_pk_{p+1}&\quad \mbox{if }\;\;k_{2p}\leq 0,\\
	\vspace*{-.3cm} & \\
	k_{p+1}\Big(\frac{k_p}{2}-\sqrt{\frac{k_p^2}{4}-k_{2p}}\Big)&\quad \mbox{if }\;\; k_{2p}>0.\\
\end{array}\right.$
\end{itemize}
\end{itemize}
\end{teo}

 \section{When Every Realization is Irreducible}\label{S3}
 
 The spectrum of a reducible matrix is the union of the spectra of its irreducible diagonal blocks. If $\sigma$ is $\mathcal{R}$ (realizable) and $\sigma$ can be partitioned into two or more spectra, each of which is $\mathcal{R}$, we then call $\sigma$ {\it partitionable}. Of course, no spectrum with a Perron root and all other eiegenvalues either negative, or in complex conjugate pairs can be partitionable. Note that if $\sigma$ is $\mathcal{R}$ and is non-partitionable, then any realization of $\sigma$ must be $\mathcal{IR}$. This can happen in a variety of ways, but the simplest is the following.
 
 \begin{teo}\label{Teo7}
 If $\sigma$ includes only one nonnegative member, then any realization of $\sigma$ is irreducible, {\it i.e.}, $\mathcal{R}$ is equivalent to $\mathcal{IR}$.
 \end{teo}

\begin{proof}
If $\sigma$ were partitionable, then each part of the partition would have a Perron root. Since there is only one possibility, $\sigma$ cannot be partitionable.
\end{proof}
 
(Of course, $\sigma$ could be partitionable and still be  $\mathcal{IR}$.) This simple fact verifies $\mathcal{IR}$ in many cases, especially in low dimensions. We suspect that ``nonnegative'' may be replaced by ``positive'' in the theorem (discussion later), but this is subtle. There are spectra that are realizable only after appending 0's (see \cite{J} for the first example). After appending the minimum number of 0's, these are also non-partitionable. We will see that if $\sigma$ is $\mathcal{IR}$, then appending any number of 0's will result in an $\mathcal{IR}$ spectrum, see Section \ref{S4}.

 \section{Appending 0's}\label{S4}

For a spectrum $\sigma$, we denote by $\sigma+b$, the spectrum that has all the elements of $\sigma$, with, additionally, the real element $b$ adjoined. Clearly, if $b\geq 0$, then $\sigma+b$ is $\mathcal{R}$ whenever $\sigma$ is $\mathcal{R}$. If $b>0$, as we shall see in the section about $n=3$,  $\sigma$ can be $\mathcal{IR}$ and $\sigma+b$ not $\mathcal{IR}$. ($\{2,-2\}$ is $\mathcal{IR}$, but $\{2,1,-2\}$ is not). What about $b=0$? We will see that $\sigma$ $\mathcal{IR}$ implies that $\sigma+0$ is $\mathcal{IR}$, a helpful fact, and a bit more. 

Applying Brauer, Theorem \ref{Brauer}, it follows that the properties  $\mathcal{IR}$ and  $\mathcal{R}_{>0}$ are monotone increasing in the Perron root and that for any $\sigma$, satisfying the reality condition, there is a positive ray that is  $\mathcal{IR}$.

There are situations in which the hypothesis above ($\sigma_{-\epsilon}$ $\mathcal{R}$) is not satisfied and 0's may still be appended to $\sigma$ while preserving $\mathcal{IR}$. Consider $\sigma=\{a,-a\}$, with $a>0$. It is $\mathcal{IR}$ as
$$
\left[\begin{array}{cc}
0 &a\\
a& 0
\end{array}
\right]
$$
is a realization. What about
$$
\{a,0,0,\dots,0,-a\}
$$
in which $n-2\,$ 0's are appended. The trace is still 0, so that $\mathcal{R}_{>0}$ is not possible, but $\mathcal{IR}$ still is possible. Consider the matrix
$$
A=\left[\begin{array}{ccccc}
0& a_1& \cdots& \cdots  &a_{n-1}\\
a_1&0 &0&  \cdots &0\\
\vdots& 0 & \ddots & \ddots & \vdots\\
\vdots& \vdots & \ddots & \ddots & 0\\
a_{n-1} & 0 & \cdots & 0 & 0
\end{array}
\right]
$$ 
in which $a_1,a_2,\dots, a_{n-1}>0$ and $\sum\limits _{j=1}^{n-1}a_j^2=a^2$. Then $\mbox{rank } A=2$, $\rm{Tr}(A^2)=2a^2$, and because of symmetry, $A$ has spectrum $\{a,0,\dots,0,-a\}$, as claimed.

This occurrence is general.
 
 \begin{teo} \label{addx}
 If $A$ is an irreducible nonnegative realization of $\sigma$ and $x\in \R$, then $\sigma+x$ is  $\mathcal{IR}$ if $|x|$ is smaller than or equal to the largest diagonal entry of $A$. In the event that $A$ is also symmetric, then $\sigma+x$ is symmetrically  $\mathcal{IR}$.
 \end{teo}
 
 \begin{proof} Suppose, without loss of generality, that the largest diagonal entry of $A=(a_{ij})$ is $a_{nn}$. Let $A_1$ be the matrix obtained by deleting last row and column of $A$. For
 $$
 S=\frac{1}{\sqrt{2}}
 \left[\begin{array}{cr}
 1 &1\\
 1 & -1
 \end{array}
 \right]
 $$
 we have
 $$
 \left[\begin{array}{cc}
 I &0\\
 0 & S^{-1}
 \end{array}
 \right]
 \left[\begin{array}{ccc|cc}
  & & & a_{1n}& 0\\
& A_1&  & \vdots& \vdots\\
& & & a_{n-1,n} &0 \\
 \hline 
 a_{n1} & \cdots& a_{n,n-1}&a_{nn}&0\\
 0 & \cdots & 0 & 0 & x
 \end{array}
 \right]
 \left[\begin{array}{cc}
 I &0\\
 0 & S
 \end{array}
 \right] = 
 \left[\begin{array}{ccc|cc}
  & & & \frac{a_{1n}}{\sqrt{2}}&  \frac{a_{1n}}{\sqrt{2}}\\
& A_1&  & \vdots& \vdots\\
& & & \frac{a_{n-1,n}}{\sqrt{2}}&\frac{a_{n-1,n}}{\sqrt{2}} \\
\vspace*{-.3cm}& & & &\\
\hline 
\vspace*{-.3cm}& & & &\\
 \frac{a_{n1} }{\sqrt{2}}& \cdots& \frac{a_{n,n-1}}{\sqrt{2}}&\frac{a_{nn}+x}{2}&\frac{a_{nn}-x}{2}\\
  \frac{a_{n1} }{\sqrt{2}} & \cdots &  \frac{a_{n,n-1}}{\sqrt{2}}&  \frac{a_{nn}-x}{2} & \frac{a_{nn}+x}{2}
 \end{array}
 \right].
 $$
 This matrix is nonnegative if $|x|\leq a_{nn}$, and irreducible, as the connectivity of the new vertex is at least that of the original vertex $n$. If $A$ were symmetric the new matrix is also. \end{proof}

 \begin{co} \label{addzeros}
 If $\sigma$ is $\mathcal{IR}$, then so is $\sigma+0$.
 \end{co}
 
 \begin{proof}
 In an irreducible realizing matrix for $\sigma$,  any diagonal entry is at least $0$, and $0\leq 0$.
 \end{proof}
 
  \begin{co} \label{addzerospol}
 If $P(x)$ is $\mathcal{IR}$, then so is $xP(x)$.
 \end{co}

  \section{The Case of $\sigma$ Real}\label{S5}
 
 We first consider the case in which all elements of $\sigma$ are nonnegative. As $\rho(\sigma)$ must have multiplicity 1 in order to have irreducibility, we have
 $$
 \lambda_1 > \lambda_2 \geq \lambda_3\geq \cdots \geq \lambda_n \geq 0.
 $$
 Since $\lambda_1>\lambda_2$, there is an $\epsilon>0$ such that $\lambda_1-\epsilon=\rho(\sigma_{-\epsilon})$ still has multiplicity 1 in $\sigma_{-\epsilon}$. Then
$$
B=\left[\begin{array}{ccccc}
\lambda_1-\epsilon& 0& \cdots& \cdots  &0\\
\lambda_1-\epsilon-\lambda_2 & \lambda_2 &0&  \cdots &0\\
\vdots& 0 & \ddots & \ddots & \vdots\\
\vdots& \vdots & \ddots & \ddots & 0\\
\lambda_1-\epsilon-\lambda_n & 0 & \cdots & 0 & \lambda_n
\end{array}
\right]\geq 0
$$ 
has constant row sums and $\mathbf{e}= (1,\dots,1)^T$ as its right eigenvector. Then $A=B+\frac{\epsilon}{n}J>0$, with $J=(1)_{n\times n}$, has spectrum $\sigma$ and $\sigma$ is $\mathcal{R}_{>0}$ and thus $\mathcal{IR}$. Notice that the same construction succeeds if some of the eigenvalues are negative but $\lambda_n > \frac{1}{n-1}(\lambda_2-\lambda_1)$.

\begin{teo} \label{Teo8}
Suppose $\sigma$ is real with elements $\lambda_1>\lambda_2 \geq \lambda_3 \geq \cdots \geq \lambda_n$ and $\lambda_n > \frac{1}{n-1}(\lambda_2-\lambda_1)$. Then $\sigma$ is $\mathcal{R}_{>0}$ and thus $\mathcal{IR}$.
\end{teo}

We next turn to  Sule\v{\i}manova \cite{Su}  type spectra $\sigma$: $\lambda_1>0, \lambda_2, \dots,\lambda_n \leq 0$ and $\rm{Tr}(\sigma)=\lambda_1+\lambda_2+\cdots+\lambda_n \geq 0$. These are known to be realizable.

\begin{teo} \label{Teo9} \label{Sule}
If $\sigma$ is of  Sule\v{\i}manova type, $\lambda_1>0\geq \lambda_2 \geq \cdots \geq  \lambda_n$, and, either $\lambda_2<0$ or  $\rm{Tr}(\sigma)>0$, then $\sigma$ is $\mathcal{IR}$.
\end{teo}

\begin{proof} Let $s_1=\rm{Tr}(\sigma)$.

 When $\lambda_2<0$ the realization of $\sigma$ given by Perfect in \cite{Pe3} 
 $$
 \left[\begin{array}{ccccc}
s_1& -\lambda_2& -\lambda_3& \cdots  &-\lambda_n\\
s_1-\lambda_2 &0&-\lambda_3&  \cdots &-\lambda_n\\\
\vdots& -\lambda_2 & \ddots & \ddots & \vdots\\
\vdots& \vdots & \ddots & \ddots & -\lambda_n\\\
s_1-\lambda_n & -\lambda_2& \cdots &-\lambda_{n-1} & 0
\end{array}
\right]
 $$
 is irreducible.
 
When  $\lambda_2=0$ and $s_1>0$ the realization of $\sigma$ given by Paparella in \cite{Pa} 
 $$
 \left[\begin{array}{ccccc}
\frac{s_1}{n}& \frac{s_1}{n}-\lambda_2& \frac{s_1}{n}-\lambda_3& \cdots  &\frac{s_1}{n}-\lambda_n\\
\frac{s_1}{n}-\lambda_2 &\frac{s_1}{n}&\frac{s_1}{n}-\lambda_3&  \cdots &\frac{s_1}{n}-\lambda_n\\
\vdots& \frac{s_1}{n}-\lambda_2 & \ddots & \ddots & \vdots\\
\vdots& \vdots & \ddots & \ddots & \frac{s_1}{n}-\lambda_n\\\
\frac{s_1}{n}-\lambda_n & \frac{s_1}{n}-\lambda_2& \cdots & \frac{s_1}{n}-\lambda_{n-1} & \frac{s_1}{n}
\end{array}
\right]
 $$
 is irreducible.
\end{proof}
 
 Now, we may show that any Sule\v{\i}manova  spectrum is $\mathcal{IR}$.
 
 \begin{co}
 Any  Sule\v{\i}manova  spectrum is $\mathcal{IR}$.
 \end{co}
 
 \begin{proof}
 If there is no  eigenvalue 0, apply Theorem \ref{Sule}. If not, use Theorem \ref{Sule}  to realize the spectrum without zeros and then apply Corollary \ref{addzeros}. \end{proof}
 
 Now, we may also append some positive eigenvalues to any Sule\v{\i}manova  spectrum.
 
 \begin{co}\label{co16}
 Let $\sigma=\{\lambda_1, \dots, \lambda_n\}$ with $\lambda_1>\lambda_2\geq \cdots \geq \lambda_t>0\geq \lambda_{t+1} \geq \cdots \geq \lambda_n$ and $\alpha=\lambda_1+\lambda_{t+1}+\cdots+\lambda_n$. If $\alpha>0$ then $\sigma$ is $\mathcal{R}_{>0}$.
 \end{co}
 
 \begin{proof} Let $\epsilon \in \left(0, \min\left\{\frac{n(\alpha-\lambda_{t+1})}{n-1}, \frac{n(\lambda_1-\lambda_{2})}{n-1}\right\}\right)$. The matrix 
 $$
B= \left[\begin{array}{ccccc|cccc}
\alpha-\epsilon& -\lambda_{t+1}& \cdots& \cdots  &-\lambda_n & & &\\
\alpha-\epsilon-\lambda_{t+1} &0 & -\lambda_{t+2}&  \cdots &-\lambda_n & & &\\
\vdots&  -\lambda_{t+1} & \ddots & \ddots & \vdots & & &\\
\vdots& \vdots & \ddots & \ddots & -\lambda_n & & &\\
\alpha-\epsilon-\lambda_n &  -\lambda_{t+1} & \cdots &  -\lambda_{n-1} &0 & & &\\
\hline 
\lambda_1-\epsilon-\lambda_2 & 0 & \cdots & \cdots & 0 & \lambda_2 & &&\\
\lambda_1-\epsilon-\lambda_3 & \vdots & \cdots & \cdots & \vdots & &\ddots   &&\\
\vdots &  \vdots  & \cdots & \cdots & \vdots & & &\ddots &\\
\lambda_1-\epsilon-\lambda_t  &  0  & \cdots & \cdots & 0& & && \lambda_t\\
\end{array}
\right]\geq 0 
$$
has constant row sums and spectrum $\sigma_\epsilon$. We can apply Brauer, Theorem \ref{Brauer}, to obtain the positive matrix $A=B+\frac{\epsilon}{n}J$ with spectrum $\sigma$.
  \end{proof}

 With $\alpha=0$, in general, we lose irreducibility, see Theorem \ref{Teo1} and Example \ref{ejem3}.

 \section{Positive Realizability vs $\mathcal{IR}$}\label{S6}
 
Of course any  $\mathcal{R}_{>0}$ spectrum is $\mathcal{IR}$, but not every $\mathcal{IR}$ spectrum $\sigma$ is $\mathcal{R}_{>0}$. But the latter is ``close'' to true. Since $\mathcal{IR}$ implies a positive (right) eigenvector, the Brauer technique may be used to show the following.
 
 \begin{teo}
 If a spectrum $\sigma$ is $\mathcal{IR}$, then $\sigma_\epsilon$ is $\mathcal{R}_{>0}$ for any $\epsilon>0$.
 \end{teo}
 
 \begin{co}\label{Corollary11}
 If $\sigma$ is a spectrum and there exists an $\epsilon>0$ such that $\sigma_{-\epsilon}$ is $\mathcal{R}$, then $\sigma$ is $\mathcal{R}_{>0}$ and thus $\mathcal{IR}$.
 \end{co}

\section{Negation Invariant Spectra}\label{S7} 

We say that a spectrum $\sigma$ is {\it negation invariant} if its element-wise negation, $-\sigma$, is the same as $\sigma$. For example, $\{a,-a\}$ (as mentioned earlier), $\{a,0,-a\}$, $\{a,b,-b,-a\}$ with $a,b>0$, or $\{a,\pm bi, -a\}$, etc.

We have noted earlier that any nonnegative $\sigma$ with $\rho(\sigma)>0$ of multiplicity 1 is 
$\mathcal{R}_{>0}$. So, the element-wise square $\sigma^2$ or nonnegative square root $\sigma^{1/2}$ is as well. If $B$ is a positive realization of $\sigma^{1/2}$, then either
$$
\left[\begin{array}{cc}
0 & B\\
B&0
\end{array}
\right]\qquad \mbox{or} \qquad \left[\begin{array}{cc}
0 & B^2\\
I&0
\end{array}
\right]
$$
is an irreducible realization of the negation invariant spectrum $\sigma \cup -\sigma$. For negation invariant $\sigma$ containing pure imaginary conjugate pairs, a similar construction may be carried out. For example, for $\sigma=\{2,-2,\pm i\}$,
$$
A=\frac{1}{2}\left[\begin{array}{cc}
3& 5\\
5&3
\end{array}
\right]
$$
is a positive realization of $\{4,-1\}$, so that
$$
\left[\begin{array}{cc}
0 & A\\
I&0
\end{array}
\right]
$$
is an irreducible realization of $\sigma$.

\section{The $\mathcal{IR}$ problem for $n=2$}\label{S8}

Since when $n=2$ the Perron root $\rho(\sigma)$ must be nonnegative, the ``other''  eigenvalue must be real and no larger in absolute value, {\it i.e.}, $\{a,b\}$ with $a \geq 0$ and $b$ real with $|b| \leq a$. However the nonnegative doubly stochastic matrix
$$
\frac{1}{2}\left[\begin{array}{cc}
a+b & a-b \\
a-b & a+b
\end{array}
\right]
$$
realizes such a spectrum. It is irreducible unless $b=a$, in which case $\rho(\sigma)$ has multiplicity 2, which of course, we assume away as we are seeking irreducibility. Thus, for $n=2$ and $\rho(\sigma)>0$ simple, $\mathcal{IR}$ is equivalent to $\mathcal{R}$.

\section{The $\mathcal{IR}$ problem for $n=3$}\label{S9}

The case $n=3$ is much richer. Spectrum $\sigma$ may consist of 1) $\rho(\sigma)$ and a conjugate pair of complex (nonreal) eigenvalues, or 2) three real eigenvalues, one of which is $\rho(\sigma)$; the other two could be a) both nonnegative, b) one nonnegative and one negative, or c) both negative.

In case 1), any realization must be irreducible, as such a spectrum is not partitionable, by Theorem \ref{Teo7}.

This leaves the real case 2). Part a) is covered by Theorem \ref{Teo8} and there is always a positive realization. Part c) is covered by Theorem \ref{Teo9}. Part b), is mostly covered by Corollary \ref{co16}.

The remaining possibility for  part b), is when $\sigma=\{a,b,-a\}$, with $a>b\geq 0$. Then $b$ must be 0 in order to be $\mathcal{IR}$, as $\sigma$ is not SS and cannot have Frobenius structure for $b>0$, see Theorem \ref{Teo1}. Otherwise $\sigma$ is not $\mathcal{IR}$.

As a conclusion, for $n=3$ and $\rho(\sigma)>0$ simple, $\mathcal{IR}$ is equivalent to $\mathcal{R}$ but $\sigma=\{a,b,-a\}$ with $b>0$. In particular, if $\rm{Tr}(\sigma)=0$, then $\mathcal{IR}$ and $\mathcal{R}$ are equivalent (see Theorem \ref{n=3} later).

\section{The $\mathcal{IR}$ problem for $n=4$}\label{S10}

\underline{The case $n=4$, $\sigma$ non-real}

Suppose $\sigma=\{a,b,c\pm di\}$ with $a>b$, $a,d>0$ is $\mathcal{R}$. If $\sigma$ is not partitionable ($b<0$ or $b \geq 0$ and $\{a,c\pm di\}$ not $\mathcal{R}$),  then $\sigma$ is $\mathcal{IR}$. Suppose then $b\geq 0$ and  $\{a,c\pm di\}$  realizable ({\it i.e.}, $a+2c,a-c-\sqrt{3}d\geq 0$, see \cite{LoLo}). It can happen:
\begin{itemize}
\item $\{a,c\pm di\}$ is $\mathcal{R}$, and $a+2c, a-c-\sqrt{3}d>0$. The matrix with constant row sums
$$
\frac{1}{3}\left[\begin{array}{cccc}
a-\epsilon+2c &a-\epsilon -c+\sqrt{3}d & a-\epsilon -c-\sqrt{3}d&0\\
a-\epsilon -c-\sqrt{3}d & a-\epsilon+2c &a-\epsilon -c+\sqrt{3}d\ & 0\\
a-\epsilon -c+\sqrt{3}d &a-\epsilon -c-\sqrt{3}d &a-\epsilon+2c & 0\\
3(a-\epsilon)-3b &0 & 0 & 3b
\end{array}
\right]
$$
has spectrum $\{a-\epsilon, b, c\pm di\}$. The application of Brauer with  $\mathbf{x}= (1,1,1,1)$ and  $\mathbf{y}=\frac{\epsilon}{4}\mathbf{x}$ gives 
$$
\frac{1}{12}\left[\begin{array}{cccc}
4(a+2c)-\epsilon &4(a-c+\sqrt{3}d)-\epsilon & 4(a-c-\sqrt{3}d)-\epsilon &3\epsilon\\
4(a-c-\sqrt{3}d)-\epsilon & 4(a+2c)-\epsilon &4(a-c+\sqrt{3}d)-\epsilon & 3\epsilon\\
4(a-c+\sqrt{3}d)-\epsilon  &4(a-c-\sqrt{3}d)-\epsilon &4(a+2c)-\epsilon & 3\epsilon\\
12(a-b)-9\epsilon &3\epsilon & 3\epsilon & 12b+3\epsilon
\end{array}
\right]
$$
with spectrum $\sigma$, which is positive for $\epsilon \in (0, \min\{4(a+2c),4(a-c-\sqrt{3}d),\frac{4}{3}(a-b)\})$.

\item $\{a,c\pm di\}$ is $\mathcal{R}$ and $a+2c=a-c-\sqrt{3}d=0$: $\sigma=\{a,b,-\frac{a}{2}\pm \frac{\sqrt{3}a}{2}i\}$. In this case, $\{a,-\frac{a}{2}\pm \frac{\sqrt{3}a}{2}i\}$ are the 3-rd roots of unity by $a$. If $b>0$, then  $\sigma$ is not $\mathcal{IR}$  by Theorem \ref{Teo1}: $\sigma$ is $\mathcal{R}$, not SS and has not Frobenius structure. If $b=0$, then $\sigma$ is $\mathcal{IR}$  by Corollary \ref{addzeros}.

\item $\{a,c\pm di\}$ is $\mathcal{R}$, $(a+2c)(a-c-\sqrt{3}d)=0$,  and $b\leq \frac{a+2c}{3}$. Then
$$
\frac{1}{3}\left[\begin{array}{ccc}
a+2c &a -c+\sqrt{3}d & a -c-\sqrt{3}d\\
a -c-\sqrt{3}d & a+2c &a -c+\sqrt{3}d\\
a-c+\sqrt{3}d &a -c-\sqrt{3}d &a+2c
\end{array}
\right]\geq 0
$$
realizes $\{a,c\pm di\}$ irreducibly and by Theorem \ref{addx} we have that $\sigma$ is $\mathcal{IR}$.

 \item $\{a,c\pm di\}$ is $\mathcal{R}$, $a+2c=0$, $a-c-\sqrt{3}d>0$ and $b>0$: $\sigma=\{a,b,-\frac{a}{2}\pm di\}$.
 
 For $b\in \left(0,\sqrt{\frac{3}{4}a^2-d^2}\,\right)$ the matrix
 $$
 \left[
 \begin{array}{cccc}
 0 & 1 & 0 & 0\\
 b^2 & 0 & 1 & 0 \\
 \frac{1}{4} (a+b)((a-2b)^2+4d^2) & 0 & 0 & 1 \\
 0 & 0 & \frac{3}{4}a^2-b^2-d^2 & b
 \end{array}
 \right]
 $$
is nonnegative, irreducible and realizes $\sigma$.

For $(b,d)$ in the  region  $\{(b,d)\;  : \;b^2+d^2\geq \frac{3}{4}a^2\geq d^2  \;,\;3b^4-12a(a^2+4d^2)b+(3a^2-4d^2)^2\geq 0\}$ the matrix 
$$
\left[\begin{array}{cccc}
\displaystyle\frac{b}{4} & 1 & 0 & 0 \\
\vspace*{-.2cm}\\
\frac{1}{2}\left(\frac{3}{4}a^2-d^2\right)+\frac{3}{16}b^2 & \displaystyle\frac{b}{4} & 1 & 0 \\
\vspace*{-.2cm}\\
\displaystyle\frac{(2a+b)((a-b)^2+4d^2)}{8} & 0 & \displaystyle\frac{b}{4} & 1 \\
\vspace*{-.2cm}\\
\displaystyle\frac{3b^4-12a(a^2+4d^2)b+(3a^2-4d^2)^2}{64} & 0 & \frac{1}{2}\left(\frac{3}{4}a^2-d^2\right)+\frac{3}{16}b^2 & \displaystyle\frac{b}{4}
\end{array}
\right]
$$
is nonnegative, irreducible and realizes $\sigma$.
 
We do not know what happen for $(b,d)$ in the  region  
 $$
 \{(b,d)\;  : \;b^2+d^2\geq \frac{3}{4}a^2\geq d^2\;,\;3b^4-12a(a^2+4d^2)b+(3a^2-4d^2)^2< 0\}.
 $$
 \item  $\{a,c\pm di\}$ is $\mathcal{R}$,  $a+2c>0$, $a-c-\sqrt{3}d=0$ and $b > \frac{a+2c}{3}\,$: $\sigma=\{a,b,c \pm \frac{a-c}{\sqrt{3}} i\}$.
 
We do not know what happen, but we conjecture that $\sigma$ is not $\mathcal{IR}$. Note that $\rm{Tr}(\sigma)= a+b+2c>\frac{4}{3}(a+2c)$. Since $k_2=\frac{1}{3}(a+2c)(a+3b+2c)>0$, by \cite[Formula (89)]{TAAMP}, the maximal diagonal entry is $b$, in which case, we have proved that any realization of $\sigma$ has the other three diagonal entries equal to $\frac{1}{3}(a+2c)$. We have also proved that $\sigma$ is not $\mathcal{IR}$ neither with constant diagonal nor with maximal diagonal entry.

\end{itemize}

\underline{The case $n=4$, $\sigma$ real}

Suppose $\lambda_1>\lambda_2\geq \lambda_3\geq \lambda_4 \geq -\lambda_1$, with $\lambda_1>0$ and $\rm{Tr}(\sigma)\geq 0$. Many particular cases are $\mathcal{IR}$, using our general results, but not all. If $\lambda_4=-\lambda_1$, then $\sigma$ must have Frobenius structure in order to be $\mathcal{IR}$. This means $\lambda_3=-\lambda_2$. In this case, $\sigma$ is realizable by
$$
\left[\begin{array}{cc}
0 & \frac{1}{2}\left[\begin{array}{cc}
\lambda_1^2+\lambda_2^2 & \lambda_1^2-\lambda_2^2\\
\vspace*{-.3cm}\\
\lambda_1^2-\lambda_2^2 & \lambda_1^2+\lambda_2^2
\end{array}\right]\\
\vspace*{-.3cm}\\
I &0 
\end{array}
\right].
$$
Otherwise, if $\lambda_4=-\lambda_1$, $\sigma$ is not $\mathcal{IR}$. In \cite{JP}, it is shown that every real 4-spectum that is realizable is either realizable by
$$
A_1=\frac{1}{2}
\left[\begin{array}{cc}
\begin{array}{cc} 
\lambda_1+\lambda_4 & \lambda_1-\lambda_4 \\
\lambda_1-\lambda_4 & \lambda_1+\lambda_4 \end{array} & 0\\
0 & \begin{array}{cc} \lambda_2+\lambda_3 & \lambda_2-\lambda_3 \\
\lambda_2-\lambda_3 & \lambda_2+\lambda_3 \end{array} 
\end{array}\right]
$$
or by
$$
A_2=\frac{1}{4}\left[\begin{array}{cccc}
S & S_1 & S_2 & S_3\\
S_1 & S & S_3 & S_2 \\
S_2 & S_3 & S & S_1 \\
S_3 & S_2 & S_1 & S
\end{array}\right],
$$
in which $S=\lambda_1+\lambda_2+\lambda_3+\lambda_4$, $S_1=\lambda_1+\lambda_2-\lambda_3-\lambda_4$, $S_2=\lambda_1-\lambda_2-\lambda_3+\lambda_4$ and $S_3=\lambda_1-\lambda_2+\lambda_3-\lambda_4$. Not all entries of $A_2$ are nonnegative, but, when all are, $A_2$ is irreducible. So, what about $A_1$? Since, we may now assume that $\{\lambda_1,\lambda_2,\lambda_3,\lambda_4\}$ is SS, we may use Brauer to give a positive realization. Since Sule\v{\i}manova spectra are covered by $A_2$, we may assume $\lambda_2>0$. Then there is an $\epsilon >0$ such that $\lambda'_1=\lambda_1-\epsilon > \lambda_2, -\lambda_4$. Now, consider $A'_1$ with $\lambda_1$ replaced by $\lambda'_1$ in $A_1$. Since $\lambda'_1>\lambda_2$, we may add positive entries in the lower left block of $A'_1$ to achieve $A''_1\geq 0$ with right Perron vector $\mathbf{e}$ and spectrum $\{\lambda_1-\epsilon, \lambda_2, \lambda_3,\lambda_4\}$. Then $A''_1+\frac{\epsilon}{4}\,\mathbf{e}\mathbf{u}^T$, with $\mathbf{u}>0$ properly normalized, has spectrum $\{\lambda_1,\lambda_2, \lambda_3,\lambda_4\}$, which is then $\mathcal{R}_{>0}$ and thus $\mathcal{IR}$.

\begin{teo}
Every real 4-spectum that is realizable is either $\mathcal{IR}$, or not SS, without Frobenius structure.
\end{teo}

This is consistent with the main conjecture.

\section{The $\mathcal{IR}$ problem with $\rm{Tr}(\sigma)=0$}\label{S11}

In this section we study the $\mathcal{IR}$ spectra $\sigma$ with trace zero
using the characteristic polynomial rather than the eigenvalues themselves. See Theorem \ref{t39} for a characterization of when $P(x)$ is $\mathcal{R}$.

We say that a weighted digraph $G=(V,E,w)$ with $V=\{1,\dots,n\}$ is a {\it simple EBL digraph} if $w(i,i+1)=1$, $1\leq i \leq n-1$, and the only possible cycles are formed by consecutive vertices and include the first or the last vertex. {\it i.e.}, they are $1,2,\dots,r,1$ or $n-(r-1),n-(r-2),\dots,n,n-(r-1)$ for $1\leq r \leq n$. Thus, the weight of each cycle coincides with the weight of its closing arc: $w(r,1)=a_{r1}$ for the {\it initial} $r$-cycles and $w(n,n+1-r)=a_{n,n+1-r}$ for the {\it final} $r$-cycles. We denote these $r$-cycles by $C_{r,1}$ and $C_{n,n+1-r}$, respectively. 
If $r=1$ we have the initial and final loops $(1,1)$ and $(n,n)$.
If $r=n$ the initial and final $n$-cycles coincide in the {\it basic $n$-cycle} $C_{n,1}$.
The adjacency matrices of these digraphs are nonnegative lower Hessenberg matrices with ones in the superdiagonal and 
some positive entries only in the first column and in the last row. Note that if these digraphs include the basic $n$-cycle {\color{red}$C_{n,1}$} or two non-disjoint cycles 
{\color{red}$C_{p,1}$} and {\color{red}$C_{n,n+1-q}$}, with
$p\geq n+1-q$, then the digraph is strongly connected, independently of the weights of the arcs. 
We distinguish in red these cycles giving the strong connection. There can be other initial or final cycles in the digraph
necessary to obtain the corresponding characteristic polynomial, but irrelevant for the strong connectivity.

In what follows we write {\color{red}$C_{p,1}(*)$} and {\color{red}$C_{n,n+1-q}(**)$} for the corresponding cycles with their weights where  $*=w(p,1)=a_{p1}$ and $**=w(n,n+1-q)=a_{n,n+1-q}$. To calculate the characteristic polynomial of the simple EBL digraphs constructed we use the Coefficient Theorem, Theorem \ref{teocof}.

\begin{teo}  \label{SFS}
If $P(x)=x^n+k_qx^{n-q}$, with $k_q<0$ and $1\leq q\leq n \geq 2$, then $P(x)$ can be realized by a strongly connected simple EBL digraph, so $P(x)$ is $\mathcal{IR}$. 
\end{teo} 

 \begin{proof}
 Let $P(x)=x^n+k_qx^{n-q}$, $1\leq q\leq n\geq 2\,$ with $k_q<0$.  We start with a digraph on the set of vertices $V=\{1,2, \dots , n\}$ formed by the basic path
$1,2, \dots , n$ with arcs of weight 1. 
 
 If $q=1$, we add loops of weight $-\frac{k_1}{2}$ in the vertices 1 and $n$. 
The characteristic polynomial of this digraph is $x^n+k_1x^{n-1}+\left(-\frac{k_1}{2}\right)^2x^{n-2}$. 

To cancel the coefficient $k_2=\left(-\frac{k_1}{2}\right)^2$ we add the arc $(2,1)$ with weight  $\left(-\frac{k_1}{2}\right)^2$, 
then the characteristic polynomial is  $x^n+k_1x^{n-1}+\left(-\frac{k_1}{2}\right)^3x^{n-3}$.

To cancel the coefficient $k_3=\left(-\frac{k_1}{2}\right)^3$ we add the arc $(3,1)$ with weight $\left(-\frac{k_1}{2}\right)^3$, 
then the characteristic polynomial is  $x^n+k_1x^{n-1}+\left(-\frac{k_1}{2}\right)^4x^{n-4}$. And so on.

Finally, to cancel the independent term $k_n=\left(-\frac{k_1}{2}\right)^n$ we add the arc $(n,1)$ with weight $\left(-\frac{k_1}{2}\right)^n$, 
then the characteristic polynomial is  $x^n+k_1x^{n-1}$. Now the digraph is strongly connected by the basic $n$-cycle {\color{red}$C_{n,1}\left(\left(-\frac{k_1}{2}\right)^n\right)$}.
The nonzero entries of the adjacency matrix of this simple EBL  digraph are: 
$a_{i+1,i}=1$,  $1\leq i \leq n-1$,  $a_{11}=a_{nn}=-\frac{k_1}{2}, a_{r1}=\left(-\frac{k_1}{2}\right)^r$,  $2\leq r \leq n$.

 If $2\leq q<n$, we add the arcs $(q,1)$ and $(n,n+1-q)$ with weights $-\frac{k_q}{2}$. If $q\geq n+1-q$, the $q$-cycles 
{\color{red}$C_{q,1}\left(-\frac{k_q}{2}\right)$} and {\color{red}$C_{n,n+1-q}\left(-\frac{k_q}{2}\right)$} are non-disjoint 
and cover all vertices of the digraph.
Then this simple EBL digraph is strongly connected and its characteristic polynomial $x^n+k_qx^{n-q}$ is $\mathcal{IR}$. 
The nonzero entries of the adjacency matrix are: 
$a_{i+1,i}=1$,  $1\leq i \leq n-1$,  $a_{q1}=a_{n,n+1-q}=-\frac{k_q}{2}$.

If $q< n+1-q$, the two $q$-cycles $C_{q,1}\left(-\frac{k_q}{2}\right)$ and $C_{n,n+1-q}\left(-\frac{k_q}{2}\right)$ are disjoint 
and the characteristic polynomial of the digraph is $x^n+k_qx^{n-q}+\left(-\frac{k_q}{2}\right)^2x^{n-2q}$.

To cancel the coefficient $k_{2q}=\left(-\frac{k_q}{2}\right)^2$ we add the arc $(2q,1)$ with weight  $\left(-\frac{k_q}{2}\right)^2$.
If the cycles $C_{2q,1}\left(\left(-\frac{k_q}{2}\right)^2\right)$ and $C_{n,n+1-q}\left(-\frac{k_q}{2}\right)$ are non-disjoint, then they cover all vertices of the digraph and the characteristic polynomial $x^n+k_qx^{n-q}$ is $\mathcal{IR}$. If the two cycles are disjoint, the characteristic polynomial is    
 $x^n+k_qx^{n-q}+\left(-\frac{k_q}{2}\right)^3x^{n-3q}$.
 
To cancel the coefficient $k_{3q}=\left(-\frac{k_q}{2}\right)^3$ we add the arc $(3q,1)$ with weight  $\left(-\frac{k_q}{2}\right)^3$. If the cycles $C_{3q,1}\left(\left(-\frac{k_q}{2}\right)^2\right)$ and $C_{n,n+1-q}\left(-\frac{k_q}{2}\right)$ are non-disjoint, then they cover all vertices of the digraph and the characteristic polynomial $x^n+k_qx^{n-q}$ is $\mathcal{IR}$. If the two cycles are disjoint, the characteristic polynomial is  $x^n+k_qx^{n-q}+\left(-\frac{k_q}{2}\right)^4x^{n-4q}$. 

And so on until $n+1-q \leq rq \leq n$, then the cycles  
{\color{red}$C_{rq,1}\left(\left(-\frac{k_q}{2}\right)^r\right)$} and 
{\color{red}$C_{n,n+1-q}\left(-\frac{k_q}{2}\right)$} are non-disjoint, they cover all vertices of the digraph and its characteristic polynomial $x^n+k_qx^{n-q}$ is $\mathcal{IR}$. 
The nonzero entries of the adjacency matrix of this simple EBL digraph are: 
$a_{i+1,i}=1$, $1\leq i \leq n-1$,  $a_{q1}=a_{n,n+1-q}=-\frac{k_q}{2}, a_{qr,1}=\left(-\frac{k_q}{2}\right)^r$,  $2\leq r \leq  \left\lfloor \frac{n}{q}\right\rfloor $.

If $q=n$, the basic $n$-cycle {\color{red}$C_{n,1}(-k_n)$} has characteristic polynomial $x^n+k_n$.
 \end{proof}

\begin{nota} \label{rem}
Note that if $q\geq p+1$, the $q$-cycles {\color{red}$C_{q,1}\left(-\frac{k_q}{2}\right)$} and {\color{red}$C_{n,n+1-q}\left(-\frac{k_q}{2}\right)$} are non-disjoint 
and cover all the vertices of the digraph.  The characteristic polynomial $x^n+k_qx^{n-q}$ is $\mathcal{IR}$, and 
the nonzero entries of its adjacency matrix are $a_{i+1,1}=1$,  $1\leq i \leq n-1$,  ${\color{red}a_{q1}=a_{n,n+1-q}=-\frac{k_q}{2}}$.  
If $q=p$ the simple EBL digraph formed by the two non-disjoint cycles {\color{red}$C_{2p,1}\left(\frac{k_p^2}{4}\right)$} and {\color{red}$C_{n,n+1-p}\left(-\frac{k_p}{2}\right)$},
and the $p$-cycle $C_{p,1}\left(-\frac{k_p}{2}\right)$ cover all the vertices of the digraph. The characteristic polynomial $x^n+k_px^{n-p}$ is $\mathcal{IR}$, and 
the nonzero entries of its adjacency matrix are
$a_{i+1,1}=1$,  $1\leq i \leq n-1$, ${\color{red}a_{2p,1}= \frac{k_p^2}{4}}$,  ${\color{red}a_{n,n+1-p}=-\frac{k_p}{2}}$, $a_{p,1}=\frac{k_p}{2}$.  
\end{nota}

The particular solutions of the cases $n=2$ to $n=5$ are as follows.

\begin{co} {\rm ($n=2$)} \label{n=2}
Let $P(x)=x^2+k_2 \neq x^2$. Then $\mathcal{R}$ and $\mathcal{IR}$ are equivalent.
\end{co}

\begin{proof} 
$P(x)$ is $\mathcal{R}$ if and only if $k_2<0$, and is $\mathcal{IR}$ by the basic 2-cycle {\color{red}$C_{2,1}\left(-k_2\right)$}, {\it i.e.}, 
by the matrix and digraph

\vspace*{-.2cm}\hspace*{4cm}$\left[\begin{array}{cc}
0 & {\color{red}1}\\
{\color{red}-k_2} & 0
\end{array}\right] $
\hspace{1cm}
\begin{minipage}{.4\hsize}
\setlength{\unitlength}{1cm}
\begin{picture}(2,2.5)
\thicklines 
\put(1,1.25){\circle*{.2}}
\put(2.5,1.25){\circle*{.2}}
\qbezier(1,1.25) (1.75,2) (2.5,1.25)
\put(1.7,.85){\small {\color{red}$1$}}
\put(1.4,1.8){\small {\color{red} $-k_2$}}
\put(1,1.25){\vector (1,0){1.4}}
\put(1.7,1.62){\vector (-1,0){0.09}}
\end{picture}
\end{minipage}\end{proof}

\begin{teo} {\rm ($n=3$)} \label{n=3}
Let $P(x)=x^3+k_2x+k_3 \neq x^3$. Then $\mathcal{R}$ and $\mathcal{IR}$ are equivalent.
\end{teo}

\begin{proof} 
$P(x)$ is $\mathcal{R}$ if and only if  $k_2,k_3\leq 0$, and is $\mathcal{IR}$ by the digraph formed by the 
 basic 3-cycle {\color{red}$C_{3,1}\left(-k_3\right)$} and the two non-disjoint 2-cycles {\color{red}$C_{2,1}\left(-\frac{k_2}{2}\right)$} and 
{\color{red}$C_{3,2}\left(-\frac{k_2}{2}\right)$}, {\it i.e.}, 
by the matrix  and digraph

\vspace*{-.2cm}\hspace*{3cm}$\left[\begin{array}{ccc}
0 & {\color{red}1} & 0 \\
{\color{red}-\frac{k_2}{2}} & 0 & {\color{red}1}\\
{\color{red}-k_3} & {\color{red}-\frac{k_2}{2}} & 0\\
\end{array}\right]$
\hspace{1cm}
\begin{minipage}{.4\hsize}
\mbox{}\hspace{-1cm}
\setlength{\unitlength}{1cm}
\begin{picture}(5,3.5)
\thicklines 
\put(1,1.25){\circle*{.2}}
\put(2.5,1.25){\circle*{.2}}
\put(4,1.25){\circle*{.2}}
\put(1,1.25){\line (1,0){3}}
\qbezier(1,1.25) (1.75,2) (2.5,1.25)
\qbezier(2.5,1.25) (3.25,2) (4,1.25)%
\put(1,1.25){\vector (1,0){1.4}}
\put(2.5,1.25){\vector (1,0){1.4}}
\put(1.7,.85){\small {\color{red}$1$}}
\put(3.2,.85){\small {\color{red}$1$}}
\put(1.5,1.85){\small {\color{red}$-\frac{k_2}{2}$}}
\put(2.8,1.85){\small {\color{red}$-\frac{k_2}{2}$}}%
\put(1.7,1.62){\vector (-1,0){0.09}}
\put(3.2,1.62){\vector (-1,0){0.09}}%
\qbezier(1,1.25) (2.5,4) (4,1.25)
\put(2.2,2.8){\small {\color{red}$-k_3$}}
\put(2.5,2.63){\vector (-1,0){0.09}}
\end{picture}
\end{minipage}
\end{proof}

\begin{teo} {\rm ($n=4$)} \label{n=4}
Let $P(x)=x^4+k_2x^2+k_3x+k_4 \neq x^4$. Then $\mathcal{R}$ and $\mathcal{IR}$ are equivalent, except if $P(x)=x^4+k_2x^2+(\frac{k_2}{2})^2$ with $k_2<0$, in which case $\mathcal{IR}$ is not possible.
\end{teo}

\begin{proof} 
$P(x)$ is $\mathcal{R}$ if and only if  $k_2, k_3\leq 0$ and $k_4\leq \frac{k_2^2}{4}$.

\noindent
a) If $k_4<0$, then $P(x)$ is $\mathcal{IR}$ by the digraph formed by the basic 4-cycle {\color{red}$C_{4,1}\left(-k_4\right)$} and the cycles 
$C_{2,1}\left(-k_2\right)$ and  $C_{3,1}\left(-k_3\right)$, with matrix  and digraph

\hspace*{2.5cm}$\left[\begin{array}{cccc}
0 & {\color{red}1} & 0 & 0\\
-k_2 & 0 & {\color{red}1} & 0\\
-k_3 & 0 & 0 & {\color{red}1}\\
{\color{red}-k_4} & 0 & 0 & 0
\end{array}\right]$
\begin{minipage}{.5\hsize}

\setlength{\unitlength}{1cm}
\begin{picture}(6,3)
\thicklines
\put(1,0.75){\circle*{.2}}
\put(2.5,0.75){\circle*{.2}}
\put(4,0.75){\circle*{.2}}
\put(5.5,0.75){\circle*{.2}}

\put(1,0.75){\line (1,0){4.5}}

\qbezier(1,0.75) (1.75,1.5) (2.5,0.75)

\put(1.6,0.33){\small {\color{red}$1$}}
\put(3.1,0.35){\small {\color{red}$1$}}
\put(4.7,0.35){\small {\color{red}$1$}}

\put(1.55,1.25){\small $-k_2$ }
\put(2.25,1.99){\small $-k_3$ }
\put(2.9,2.65){\small {\color{red}$-k_4$}}

\qbezier (1,0.75) (2.5,3) (4,0.75)
\qbezier (1,0.75) (3.25,4.25) (5.5,0.75)

\put(2.36,0.75){\vector (1,0){0.09}}
\put(3.86,0.75){\vector (1,0){0.09}}
\put(5.36,0.75){\vector (1,0){0.09}}
\put(1.7,1.12){\vector (-1,0){0.09}}
\put(2.5,1.87){\vector (-1,0){0.09}}
\put(3.2,2.49){\vector (-1,0){0.09}}
\end{picture}
\vspace{.5em}

\end{minipage}

\noindent b1) If $k_4=0$, and $k_2,k_3<0$, then $P(x)$ is $\mathcal{IR}$ by the digraph formed by the two non-disjoint cycles 
{\color{red}$C_{2,1}\left(-k_2\right)$} and  {\color{red}$C_{4,2}\left(-k_3\right)$}, with matrix  and digraph

\hspace*{2.5cm}$\left[\begin{array}{cccc}
0 & {\color{red}1} & 0 & 0\\
{\color{red}-k_2} & 0 & {\color{red}1} & 0\\
0 & 0 & 0 & {\color{red}1}\\
0 & {\color{red}-k_3}  &0 & 0
\end{array}\right]$
\begin{minipage}{.5\hsize}

\setlength{\unitlength}{1cm}
\begin{picture}(6,3)
\thicklines
\put(1,0.75){\circle*{.2}}
\put(2.5,0.75){\circle*{.2}}
\put(4,0.75){\circle*{.2}}
\put(5.5,0.75){\circle*{.2}}

\put(1,0.75){\line (1,0){4.5}}

\qbezier(1,0.75) (1.75,1.5) (2.5,0.75)

\put(1.6,0.33){\small {\color{red}$1$}}
\put(3.1,0.35){\small {\color{red}$1$}}
\put(4.7,0.35){\small {\color{red}$1$}}

\put(1.5,1.3){\small {\color{red}$-k_2$}}
\put(3.7,2.05){\small {\color{red}$-k_3$}}

\qbezier (2.5,0.75) (4,3) (5.5,0.75)

\put(2.36,0.75){\vector (1,0){0.09}}
\put(3.86,0.75){\vector (1,0){0.09}}
\put(5.36,0.75){\vector (1,0){0.09}}
\put(1.7,1.12){\vector (-1,0){0.09}}
\put(4,1.87){\vector (-1,0){0.09}}
\end{picture}
\vspace{.5em}

\end{minipage}

\noindent b2) If $k_4=0$, and $k_2=0$ or $k_3=0$, the respective polynomials are $x^2(x^2+k_2)$ or  $x(x^3+k_3)$ and by Theorem \ref{SFS} the polynomial $P(x)$ is $\mathcal{IR}$.

\noindent c1)  If $0<k_4\leq \frac{k_2^2}{4}$ and $k_3<0$, then $P(x)$ is $\mathcal{IR}$ by the digraph formed by the two non-disjoint cycles 
 {\color{red}$C_{3,1}\left(-k_3\right)$} and {\color{red}$C_{4,3}\left(a_{43}\right)$}, and the cycle $C_{2,1}\left(a_{21}\right)$, with matrix and digraph

$\left[\begin{array}{cccc}
0 & {\color{red}1} & 0 & 0\\
a_{21} & 0 & {\color{red}1} & 0\\
{\color{red}-k_3} & 0 & 0 & {\color{red}1}\\
0 & 0 & {\color{red}a_{43}} & 0
\end{array}\right]$  \hspace{0.1cm}{where}\hspace{0.1cm}
${\left\{
\begin{array}{rl}
a_{21} & =-\frac{k_2}{2}-\sqrt{\frac{k_2^2}{4}-k_4},  \\
{\color{red}a_{43}} & ={\color{red}-\frac{k_2}{2}+\sqrt{\frac{k_2^2}{4}-k_4}},
\end{array}
\right.}$
\begin{minipage}{.5\hsize}

\setlength{\unitlength}{1cm}
\begin{picture}(6,3)
\thicklines
\put(1,0.75){\circle*{.2}}
\put(2.5,0.75){\circle*{.2}}
\put(4,0.75){\circle*{.2}}
\put(5.5,0.75){\circle*{.2}}

\put(1,0.75){\line (1,0){4.5}}

\qbezier(1,0.75) (1.75,1.5) (2.5,0.75)
\qbezier(4,0.75) (4.75,1.5) (5.5,0.75)

\put(1.6,0.33){\small {\color{red}$1$}}
\put(3.1,0.35){\small {\color{red}$1$}}
\put(4.7,0.35){\small {\color{red}$1$}}

\put(1.65,1.3){\small $a_{21}$}
\put(4.6,1.3){\small {\color{red}$a_{43}$}}
\put(2.2,2.05){\small {\color{red}$-k_{3}$}}

\qbezier (1,0.75) (2.5,3) (4,0.75)%

\put(2.36,0.75){\vector (1,0){0.09}}
\put(3.86,0.75){\vector (1,0){0.09}}
\put(5.36,0.75){\vector (1,0){0.09}}
\put(1.7,1.12){\vector (-1,0){0.09}}
\put(4.7,1.12){\vector (-1,0){0.09}}
\put(2.5,1.87){\vector (-1,0){0.09}}
\end{picture}
\vspace{.5em}

\end{minipage}

\noindent c2)  If $0<k_4<\frac{k_2^2}{4}$ and $k_3=0$, then 
$P(x)=\left(x^2-\left(-\frac{k_2}{2}-\sqrt{\frac{k_2^2}{4}-k_4}\right)\right)\left(x^2-\left(-\frac{k_2}{2}+\sqrt{\frac{k_2^2}{4}-k_4}\right)\right)$ is $\mathcal{IR}$ by the matrix and digraph

\hspace*{2.5cm}$\left[\begin{array}{cccc}
0 & {\color{red}1} & 0 & 0\\
-\frac{k_2}{2} & 0 & {\color{red}1} & 0\\
0 & 0 & 0 & {\color{red}1}\\
{\color{red}\frac{k_2^2}{4}-k_4} & 0 & -\frac{k_2}{2} & 0
\end{array}\right]$
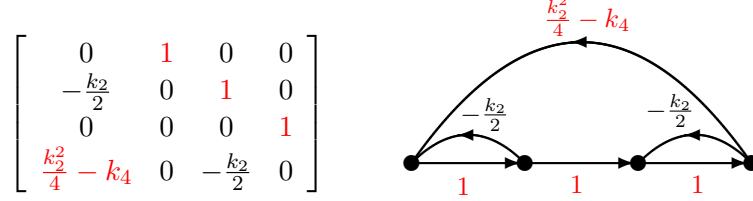
\begin{minipage}{.5\hsize}

\setlength{\unitlength}{1cm}
\begin{picture}(6,3)
\thicklines
\put(1,0.75){\circle*{.2}}
\put(2.5,0.75){\circle*{.2}}
\put(4,0.75){\circle*{.2}}
\put(5.5,0.75){\circle*{.2}}

\put(1,0.75){\line (1,0){4.5}}

\qbezier(1,0.75) (1.75,1.5) (2.5,0.75)
\qbezier(4,0.75) (4.75,1.5) (5.5,0.75)

\put(1.6,0.33){\small {\color{red}$1$}}
\put(3.1,0.35){\small {\color{red}$1$}}
\put(4.7,0.35){\small {\color{red}$1$}}

\put(1.65,1.3){\small $-\frac{k_2}{2}$}
\put(4.1,1.35){\small $-\frac{k_2}{2}$}
\put(2.75,2.55){\small {\color{red}$\frac{k_2^2}{4}-k_4$}}

\qbezier (1,0.75) (3.25,3.95) (5.5,0.75)

\put(2.36,0.75){\vector (1,0){0.09}}
\put(3.86,0.75){\vector (1,0){0.09}}
\put(5.36,0.75){\vector (1,0){0.09}}
\put(1.7,1.12){\vector (-1,0){0.09}}
\put(4.7,1.12){\vector (-1,0){0.09}}
\put(3.2,2.35){\vector (-1,0){0.09}}
\end{picture}
\vspace{.5em}

\end{minipage}

\noindent c3)  If $0<k_4=\frac{k_2^2}{4}$ and $k_3=0$, then $P(x)=\left(x^2+\frac{k_2}{2}\right)^2$ does not have simple Perron root, 
so is not $\mathcal{IR}$.\end{proof}

\begin{teo} {\rm ($n=5$)} \label{n=5}
Let $P(x)=x^5+k_2x^3+k_3x^2+k_4x+k_5 \neq x^5$. Then $\mathcal{R}$ and $\mathcal{IR}$ are equivalent, except if 

\noindent 1) $k_2, k_3<0, k_4=0$ and $k_5=k_2k_3$; or

\noindent 2) $k_2,k_3<0, 0<k_4<\frac{k_2^2}{4}$ and $k_5=k_{3}\Big(\frac{k_2}{2}-\sqrt{\frac{k_2^2}{4}-k_{4}}\Big)$; or

\noindent 3) $k_2<0, k_3=0, k_4=\frac{k_2^2}{4}$ and $k_5=0$;

\noindent in which cases $\mathcal{IR}$ is not possible.

\end{teo}

\begin{proof} 
$P(x)$ is $\mathcal{R}$ if and only if  $k_2,k_3\leq 0$, $k_4\leq \frac{k_2^2}{4}$ and 
$k_{5}\leq\left\{
\begin{array}{ll}
	k_2k_{3}& \;\; \mbox{if }\;\;k_{4}\leq 0,\\
	\vspace*{-.3cm} & \\
	k_{3}\Big(\frac{k_2}{2}-\sqrt{\frac{k_2^2}{4}-k_{4}}\Big)&\;\; \mbox{if }\;\; k_{4}>0.\\
\end{array}\right.$

\noindent a1) If $k_4< 0$, and $k_2<0$ or $k_5<k_2k_3$, then $P(x)$ is $\mathcal{IR}$ by the digraph formed by the non-disjoint cycles
{\color{red}$C_{4,1}\left(-k_4\right)$} and {\color{red}$C_{5,4}\left(-k_2\right)$}, or 
the basic $5$-cycle {\color{red}$C_{5,1}\left(k_2k_3-k_5\right)$}, 
and the cycle $C_{3,1}\left(-k_3\right)$ with matrix and digraph

\hspace*{2.2cm}$\left[\begin{array}{ccccc}
0 & {\color{red}1} & 0 & 0 & 0\\
0 & 0 & {\color{red}1} & 0 & 0\\
-k_3 & 0 & 0 & {\color{red}1} & 0\\
{\color{red}-k_4} & 0 & 0 & 0 & {\color{red}1} \\
{\color{red}k_2k_3-k_5} & 0 & 0 & {\color{red}-k_2} & 0
\end{array}\right]$
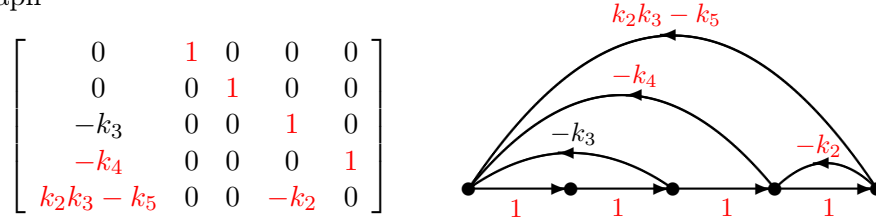
 \begin{minipage}{.35\hsize}
\setlength{\unitlength}{.9cm}
\vspace*{.3cm}
\begin{picture}(6,3)
\thicklines
\put(1,0.75){\circle*{.2}}
\put(2.5,0.75){\circle*{.2}}
\put(4,0.75){\circle*{.2}}
\put(5.5,0.75){\circle*{.2}}
\put(7,0.75){\circle*{.2}}

\put(1,0.75){\line (1,0){5.9}}

\qbezier(1,0.75) (2.5,1.8) (4,0.75)
\qbezier(1,0.75) (3.25,3.5) (5.5,0.75)
\qbezier(5.5,0.75) (6.25,1.5) (7,0.75)%


\put(1.6,0.33){\small {\color{red}$1$}}
\put(3.1,0.35){\small {\color{red}$1$}}
\put(4.7,0.35){\small {\color{red}$1$}}
\put(6.2,0.35){\small {\color{red}$1$}}

\put(3.1,2.3){\small {\color{red}$-k_4$}}
\put(5.8,1.28){\small {\color{red}$-k_2$}}
\put(2.2,1.45){\small $-k_3$ }
\put(3.1,3.2){\small {\color{red}$k_2k_3-k_5$}}

\qbezier (1,0.75) (4,5.25) (7,0.75)

\put(2.36,0.75){\vector (1,0){0.09}}
\put(3.86,0.75){\vector (1,0){0.09}}
\put(5.36,0.75){\vector (1,0){0.09}}
\put(6.86,0.75){\vector (1,0){0.09}}
\put(2.4,1.28){\vector (-1,0){0.09}}
\put(3.4,2.12){\vector (-1,0){0.09}}
\put(6.15,1.12){\vector (-1,0){0.09}}%
\put(3.9,3.01){\vector (-1,0){0.09}}
\end{picture}
\end{minipage}

\noindent a2) If $k_3,k_4<0$ and  $k_2=k_5=0$, then $P(x)$ is $\mathcal{IR}$ by the digraph formed by  
the two non-disjoint cycles {\color{red}$C_{5,2}\left(-k_4\right)$} and {\color{red}$C_{3,1}\left(-k_3\right)$} with matrix and digraph

\hspace*{2.2cm}$\left[\begin{array}{ccccc}
0 & {\color{red}1} & 0 & 0 & 0\\
0 & 0 & {\color{red}1} & 0 & 0\\
{\color{red}-k_3}& 0 & 0 & {\color{red}1} & 0\\
0 & 0 & 0 & 0 & {\color{red}1} \\
0 & {\color{red}-k_4} & 0 & 0 & 0
\end{array}\right]$
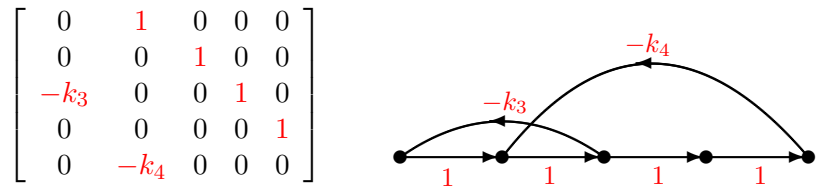
\begin{minipage}{.35\hsize}
\setlength{\unitlength}{.9cm}
\vspace*{.3cm}
\begin{picture}(6,3)
\thicklines
\put(1,0.75){\circle*{.2}}
\put(2.5,0.75){\circle*{.2}}
\put(4,0.75){\circle*{.2}}
\put(5.5,0.75){\circle*{.2}}
\put(7,0.75){\circle*{.2}}

\put(1,0.75){\line (1,0){5.9}}

\qbezier(1,0.75) (2.5,1.8) (4,0.75)
\qbezier(2.5,0.75) (4.65,3.5) (7,0.75)%


\put(1.6,0.33){\small {\color{red}$1$}}
\put(3.1,0.35){\small {\color{red}$1$}}
\put(4.7,0.35){\small {\color{red}$1$}}
\put(6.2,0.35){\small {\color{red}$1$}}

\put(4.3,2.3){\small {\color{red}$-k_4$}}
\put(2.2,1.45){\small {\color{red}$-k_3$}}


\put(2.36,0.75){\vector (1,0){0.09}}
\put(3.86,0.75){\vector (1,0){0.09}}
\put(5.36,0.75){\vector (1,0){0.09}}
\put(6.86,0.75){\vector (1,0){0.09}}
\put(2.4,1.28){\vector (-1,0){0.09}}
\put(4.55,2.13){\vector (-1,0){0.09}}
\end{picture}
\end{minipage}

\noindent a3) If $k_4<0$, and $k_2=k_3=k_5=0$, then $P(x)=x(x^4+k_4)$  and, by Theorem \ref{SFS}, is $\mathcal{IR}$.

\noindent b1) If $k_4=0$ and $k_5<k_2k_3$, then $P(x)$ is $\mathcal{IR}$ by the digraph formed by the basic $5$-cycle {\color{red}$C_{5,1}\left(k_2k_3-k_5\right)$}, 
and the cycles $C_{3,1}\left(-k_3\right)$ and $C_{5,4}\left(-k_2\right)$ with matrix and digraph

\hspace*{2.2cm}$\left[\begin{array}{ccccc}
0 & {\color{red}1} & 0 & 0 & 0\\
0 & 0 & {\color{red}1} & 0 & 0\\
-k_3 & 0 & 0 & {\color{red}1} & 0\\
0 & 0 & 0 & 0 & {\color{red}1} \\
{\color{red}k_2k_3-k_5} & 0 & 0 & -k_2 & 0
\end{array}\right]$
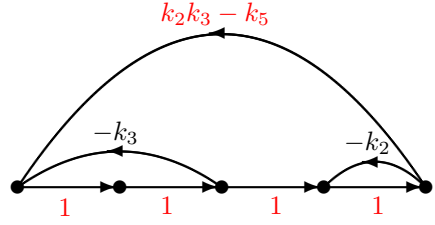
\begin{minipage}{.35\hsize}
\setlength{\unitlength}{.9cm}
\vspace*{.7cm}
\begin{picture}(6,3)
\thicklines
\put(1,0.75){\circle*{.2}}
\put(2.5,0.75){\circle*{.2}}
\put(4,0.75){\circle*{.2}}
\put(5.5,0.75){\circle*{.2}}
\put(7,0.75){\circle*{.2}}

\put(1,0.75){\line (1,0){5.9}}

\qbezier(1,0.75) (2.5,1.8) (4,0.75)
\qbezier(5.5,0.75) (6.25,1.5) (7,0.75)


\put(1.6,0.33){\small {\color{red}$1$}}
\put(3.1,0.35){\small {\color{red}$1$}}
\put(4.7,0.35){\small {\color{red}$1$}}
\put(6.2,0.35){\small {\color{red}$1$}}

\put(5.8,1.3){\small $-k_2$}
\put(2.1,1.45){\small $-k_3$}
\put(3.1,3.2){\small {\color{red}$k_2k_3-k_5$}}

\qbezier (1,0.75) (4,5.25) (7,0.75)

\put(2.36,0.75){\vector (1,0){0.09}}
\put(3.86,0.75){\vector (1,0){0.09}}
\put(5.36,0.75){\vector (1,0){0.09}}
\put(6.86,0.75){\vector (1,0){0.09}}
\put(2.4,1.28){\vector (-1,0){0.09}}
\put(6.15,1.12){\vector (-1,0){0.09}}%
\put(3.9,3.01){\vector (-1,0){0.09}}
\end{picture}
\end{minipage}

\noindent b2) If $k_4=0$, $k_5=k_2k_3$, and $k_2,k_3<0$, then  $P(x)=(x^2+k_2)(x^3+k_3)$ does not have Frobenius structure so, by Theorem \ref{Teo1}, is not $\mathcal{IR}$. This is the exceptional case 1).

\noindent b3) If $k_4=0$, $k_5=k_2k_3$, and $k_2=0$ or $k_3=0$, the respective polynomials $x^2(x^3+k_3)$ or $x^3(x^2+k_2)$  are $\mathcal{IR}$ by Theorem \ref{SFS}.

\noindent c1)  If $0<k_4\leq \frac{k_2^2}{4}$ and $k_5<k_3\left(\frac{k_2}{2}-\sqrt{\frac{k_2^2}{4}-k_4}\right)$, then $P(x)$ is $\mathcal{IR}$ by the digraph formed by
the basic $5$-cycle  {\color{red}$C_{5,1}\left(a_{51}\right)$}, and the cycles $C_{2,1}\left(a_{21}\right)$,  {$C_{3,1}\left(-k_3\right)$} and $C_{5,4}\left(a_{54}\right)$, 
with matrix and digraph
 
\hspace*{-1.3cm}$\left[\begin{array}{ccccc}
0 & {\color{red}1} & 0 & 0 & 0\\
a_{21} & 0 & {\color{red}1} & 0 & 0\\
-k_3 & 0 & 0 & {\color{red}1} & 0\\
0 & 0 & 0 & 0 & {\color{red}1} \\
{\color{red}a_{51}} & 0 & 0 & a_{54} & 0
\end{array}\right]$ \hspace{.1cm}{with}\hspace{.1cm}
${\left\{
\begin{array}{rl}
a_{21} & =-\frac{k_2}{2}-\sqrt{\frac{k_2^2}{4}-k_4},  \\
a_{54} & =-\frac{k_2}{2}+\sqrt{\frac{k_2^2}{4}-k_4},\\
{\color{red}a_{51}} & ={\color{red} k_3\left(\frac{k_2}{2}-\sqrt{\frac{k_2^2}{4}-k_4}\right)-k_5}>0, \\
\end{array}
\right.}$
\hspace*{-.8cm}\begin{minipage}{.35\hsize}
\setlength{\unitlength}{.9cm}
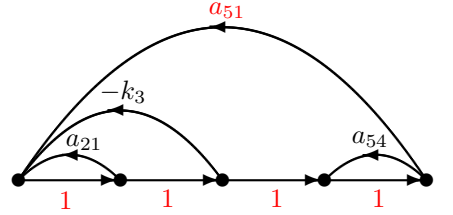
\vspace*{.7cm}
\begin{picture}(6,3)
\thicklines
\put(1,0.75){\circle*{.2}}
\put(2.5,0.75){\circle*{.2}}
\put(4,0.75){\circle*{.2}}
\put(5.5,0.75){\circle*{.2}}
\put(7,0.75){\circle*{.2}}

\put(1,0.75){\line (1,0){5.9}}

\qbezier(1,0.75) (2.5,2.8) (4,0.75)
\qbezier(5.5,0.75) (6.25,1.5) (7,0.75)
\qbezier(1,0.75) (1.75,1.5) (2.5,0.75)


\put(1.6,0.33){\small {\color{red}$1$}}
\put(3.1,0.35){\small {\color{red}$1$}}
\put(4.7,0.35){\small {\color{red}$1$}}
\put(6.2,0.35){\small {\color{red}$1$}}

\put(1.7,1.25){\small $a_{21}$ }
\put(5.9,1.3){\small $a_{54}$}
\put(2.2,1.95){\small $-k_3$}
\put(3.8,3.2){\small {\color{red}$a_{51}$}}

\qbezier (1,0.75) (4,5.25) (7,0.75)

\put(1.73,1.125){\vector (-1,0){0.09}}
\put(2.36,0.75){\vector (1,0){0.09}}
\put(3.86,0.75){\vector (1,0){0.09}}
\put(5.36,0.75){\vector (1,0){0.09}}
\put(6.86,0.75){\vector (1,0){0.09}}
\put(2.4,1.78){\vector (-1,0){0.09}}%
\put(6.15,1.12){\vector (-1,0){0.09}}
\put(3.9,3.01){\vector (-1,0){0.09}}
\end{picture}
\end{minipage}

\noindent c2) Let $0<k_4< \frac{k_2^2}{4}$, $k_3<0$ and $k_5=k_3\left(\frac{k_2}{2}-\sqrt{\frac{k_2^2}{4}-k_4}\right)$.

\noindent Since $k_{4}>0$, there exists $m_0\geq -\frac{k_2}{2}$ such that $k_{4}=m_0(-k_2-m_0)$. Then $m_0=-\frac{k_2}{2}+\sqrt{\frac{k_2^2}{4}-k_{4}}$.

\noindent This situation correponds to a digraph with two disjoint $2$-cycles with weights $m_0$ and $-k_2-m_0$ and without $4$-cycles. 
And this is the optimum situation (see the proof of Theorem \ref{t39} in \cite{TAAMP}).

The value $k_{3}\left(\frac{k_2}{2}-\sqrt{\frac{k_2^2}{4}-k_{4}}\right)$ is the maximum value of the coefficient $k_{5}$ and so it is obtained with
the absence of $5$-cycles and with a single $3$-cycle disjoint with the $2$-cycle of maximum weight $m_0$. 
For that $k_{5}=k_{3}\left(\frac{k_2}{2}-\sqrt{\frac{k_2^2}{4}-k_{4}}\right)$. This digraph is not strongly connected and so there is no $\mathcal{IR}$.
This is the exceptional case 2).

\noindent c3)  If $0<k_4<\frac{k_2^2}{4}$ and $k_3=k_5=0$, then $P(x)=x\left(x^2-\left(-\frac{k_2}{2}-\sqrt{\frac{k_2^2}{4}-k_4}\right)\right)\left(x^2-\left(-\frac{k_2}{2}+\sqrt{\frac{k_2^2}{4}-k_4}\right)\right)$  is $\mathcal{IR}$ by the digraph formed by the two non-disjoint cycles  {\color{red}$C_{4,1}\left(\frac{k_2^2}{4}-k_4\right)$} and 
{\color{red}$C_{5,4}\left(-\frac{k_2}{2}\right)$}, and the cycle $C_{2,1}\left(-\frac{k_2}{2}\right)$ with matrix and digraph

\hspace*{2.2cm}$\left[\begin{array}{ccccc}
0 & {\color{red}1} & 0 & 0 & 0\\
-\frac{k_2}{2} & 0 & {\color{red}1} & 0 & 0\\
0 & 0 & 0 & {\color{red}1} & 0\\
{\color{red}\frac{k_2^2}{4}-k_4} & 0 & 0 & 0 & {\color{red}1}\\
0 & 0 & 0 & {\color{red}-\frac{k_2}{2}} & 0
\end{array}\right]$
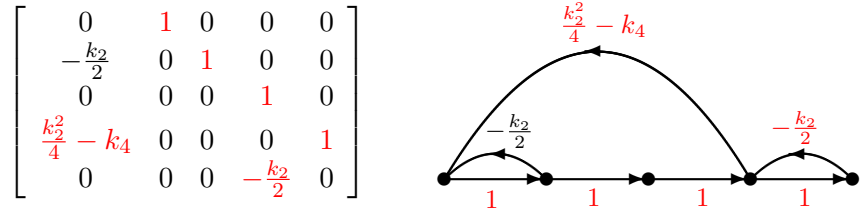
\begin{minipage}{.35\hsize}
\setlength{\unitlength}{.9cm}
\vspace*{.7cm}
\begin{picture}(6,3)
\thicklines
\put(1,0.75){\circle*{.2}}
\put(2.5,0.75){\circle*{.2}}
\put(4,0.75){\circle*{.2}}
\put(5.5,0.75){\circle*{.2}}
\put(7,0.75){\circle*{.2}}

\put(1,0.75){\line (1,0){5.9}}

\qbezier(1,0.75) (3.25,4.5) (5.5,0.75)%
\qbezier(5.5,0.75) (6.25,1.5) (7,0.75)
\qbezier(1,0.75) (1.75,1.5) (2.5,0.75)


\put(1.6,0.33){\small {\color{red}$1$}}
\put(3.1,0.35){\small {\color{red}$1$}}
\put(4.7,0.35){\small {\color{red}$1$}}
\put(6.2,0.35){\small {\color{red}$1$}}

\put(1.6,1.35){\small $-\frac{k_2}{2}$ }
\put(5.8,1.4){\small {\color{red}$-\frac{k_2}{2}$}}
\put(2.7,2.9){\small {\color{red}$\frac{k_2^2}{4}-k_4$}}


\put(1.75,1.12){\vector (-1,0){0.09}}
\put(2.36,0.75){\vector (1,0){0.09}}
\put(3.86,0.75){\vector (1,0){0.09}}
\put(5.36,0.75){\vector (1,0){0.09}}
\put(6.86,0.75){\vector (1,0){0.09}}
\put(3.15,2.63){\vector (-1,0){0.09}}
\put(6.15,1.12){\vector (-1,0){0.09}}
\end{picture}
\end{minipage}

\noindent c4) If $0<k_4= \frac{k_2^2}{4}$, $k_3<0$ and $k_5=\frac{k_2k_3}{2}$, then $P(x)$ is $\mathcal{IR}$ 
by the digraph formed by the two non-disjoint 3-cycles {\color{red}$C_{3,1}\left(-\frac{k_3}{2}\right)$} and {\color{red}$C_{5,3}\left(-\frac{k_3}{2}\right)$}, 
and the 2-cycles $C_{2,1}\left(-\frac{k_2}{2}\right)$ and $C_{5,4}\left(-\frac{k_2}{2}\right)$ with matrix and digraph

\hspace*{2.2cm}$\left[\begin{array}{ccccc}
0 & {\color{red}1} & 0 & 0 & 0\\
-\frac{k_2}{2} & 0 & {\color{red}1} & 0 & 0\\
{\color{red}-\frac{k_3}{2}} & 0 & 0 & {\color{red}1} & 0\\
0 & 0 & 0 & 0 & {\color{red}1} \\
0 & 0 & {\color{red}-\frac{k_3}{2}} & -\frac{k_2}{2} & 0
\end{array}\right]$
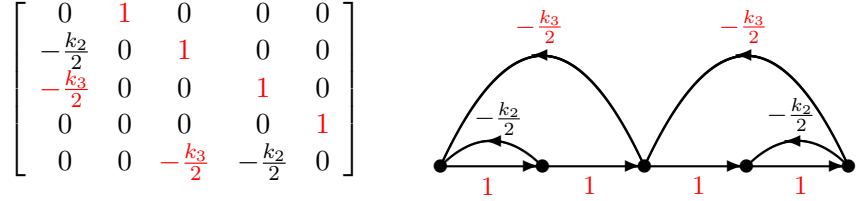
\begin{minipage}{.35\hsize}
\setlength{\unitlength}{.9cm}
\vspace*{.7cm}
\begin{picture}(6,3)
\thicklines
\put(1,0.75){\circle*{.2}}
\put(2.5,0.75){\circle*{.2}}
\put(4,0.75){\circle*{.2}}
\put(5.5,0.75){\circle*{.2}}
\put(7,0.75){\circle*{.2}}

\put(1,0.75){\line (1,0){5.9}}

\qbezier(5.5,0.75) (6.25,1.5) (7,0.75)
\qbezier(1,0.75) (1.75,1.5) (2.5,0.75)


\put(1.6,0.33){\small {\color{red}$1$}}
\put(3.1,0.35){\small {\color{red}$1$}}
\put(4.7,0.35){\small {\color{red}$1$}}
\put(6.2,0.35){\small {\color{red}$1$}}

\put(1.5,1.35){\small $-\frac{k_2}{2}$ }
\put(5.8,1.4){\small $-\frac{k_2}{2}$}
\put(2.1,2.7){\small {\color{red}$-\frac{k_3}{2}$}}
\put(5.1,2.7){\small {\color{red}$-\frac{k_3}{2}$}}

\qbezier (1,0.75) (2.5,4) (4,0.75)
\qbezier (4,0.75) (5.5,4) (7,0.75)

\put(1.75,1.12){\vector (-1,0){0.09}}
\put(2.36,0.75){\vector (1,0){0.09}}
\put(3.86,0.75){\vector (1,0){0.09}}
\put(5.36,0.75){\vector (1,0){0.09}}
\put(6.86,0.75){\vector (1,0){0.09}}
\put(2.43,2.38){\vector (-1,0){0.09}}%
\put(5.43,2.38){\vector (-1,0){0.09}}
\put(6.15,1.12){\vector (-1,0){0.09}}
\end{picture}
\end{minipage}

\noindent c5)  If $0<k_4=\frac{k_2^2}{4}$ and $k_3=k_5=0$, then $P(x)=x\left(x^2+\frac{k_2}{2}\right)^2$ does not have simple Perron root, so is not $\mathcal{IR}$. This is the exceptional case 3).\end{proof}

We give the general $\mathcal{IR}$ solution for the polynomial $P(x)=x^n+k_px^{n-p}+ \dots +k_{n-1}x+k_n$ in terms of the parity of $n$ 
and of the sign of the coefficient $k_{2p}$ (see the proof of Theorem 39 in \cite{TAAMP}).
In the next two theorems, for each $\mathcal{IR}$ polynomial $P(x)$, we give the initial and final cycles of the 
strongly connected simple EBL digraph realizing $P(x)$,
and, to describe the corresponding irreducible matrix, we only give the possible nonzero entries in the first column and in the last row of the adjacency matrix corresponding 
to the closing arcs of the initial and final cycles of the mentioned digraph realizing each $\mathcal{IR}$ polynomial $P(x)$.

\begin{teo} {\rm ($n=2p$)} \label{n=2p}
Let $P(x)=x^{2p}+k_px^p+k_{p+1}x^{p-1}+ \dots +k_{2p-1}x+k_{2p} \neq x^{2p}$ with $p\geq 3$. Then $\mathcal{R}$ and $\mathcal{IR}$ are equivalent, except if $k_p<0, 
k_{p+1} = \dots =k_{2p-1}=0$ and $k_{2p}=\frac{k_p^2}{4}$, in which case $\mathcal{IR}$ is not possible.
\end{teo}

\begin{proof} From Theorem \ref{t39}, $P(x)$ is $\mathcal{R}$ if and only if  $k_p\leq 0, \, k_{p+1}\leq 0, \dots , k_{2p-1}\leq 0$ and $k_{2p}\leq \frac{k_o^2}{4}$.

\noindent
a) If $k_{2p}<0$, then $P(x)$ is $\mathcal{IR}$ by the digraph formed by the basic $2p$-cycle {\color{red}$C_{2p,1}\left(-k_{2p}\right)$} and the $r$-cycles 
$C_{r,1}\left(-k_r\right), \, r\in \{p, p+1, \dots , 2p-1\}$. Then {\color{red}$a_{2p,1}=-k_{2p}>0$} and 
$a_{r,1}=-k_r\geq 0, \, r\in \{p, p+1, \dots , 2p-1\}$.

\noindent b1) If $k_{2p}=0$, and there exist indices $q_1$ and  $q_2$, with $p\leq  q_1 < q_2 \leq 2p-1$, such that $k_{q_1}, k_{q_2}<0$, 
then $P(x)$ is $\mathcal{IR}$ by the digraph formed with the two non-disjoint cycles {\color{red}$C_{q_1,1}\left(-k_{q_1}\right)$} and 
{\color{red}$C_{2p,2p+1-q_2}\left(-k_{q_2}\right)$}, and the $r$-cycles  $C_{r,1}\left(-k_r\right), \, r\in \{p, p+1, \dots , 2p-1\}-\{q_1, q_2\}$. 
Then ${\color{red}a_{q_1,1}=-k_{q_1}}, 
{\color{red}a_{2p,2p+1-q_2}=-k_{q_2}>0}$, and  $a_{r,1}=-k_r\geq 0, \, r\in \{p, p+1, \dots , 2p-1\}-\{q_1, q_2\}$.

\noindent b2) If $k_{2p}=0$, and there exists only an index $q$, with $p\leq  q  \leq 2p-1$, such that $k_{q}<0$, then the respective polynomials 
$x^{p}(x^p+k_p), x^{p-1}(x^{p+1}+k_{p+1}), \dots , x(x^{2p-1}+k_{2p-1})$, by Theorem \ref{SFS}, are $\mathcal{IR}$. 

\noindent c1) If $0<k_{2p}\leq \frac{k_3^2}{4}$, and there exists an index $q$, with $p+1\leq  q  \leq 2p-1$, such that $k_{q}<0$,
then $P(x)$ is $\mathcal{IR}$ by the digraph formed with the two non-disjoint cycles {\color{red}$C_{q,1}\left(-k_{q}\right)$} and 
{\color{red}$C_{2p,p+1}\left(-\frac{k_p}{2}+\sqrt{\frac{k_p^2}{4}-k_{2p}}\right)$}, and the cycles  $C_{p,1}\left(-\frac{k_p}{2}-\sqrt{\frac{k_p^2}{4}-k_{2p}}\right)$ and $C_{r,1}\left(-k_r\right), \, r\in \{p+1, \dots , 2p-1\}-\{q\}$.
Then {\color{red}$a_{q,1}=-k_{q}>0$}, {\color{red}$a_{2p,p+1}=-\frac{k_p}{2}+\sqrt{\frac{k_p^2}{4}-k_{2p}}>0$},
 $a_{p,1}=-\frac{k_p}{2}-\sqrt{\frac{k_p^2}{4}-k_{2p}}>0$ and $a_{r,1}=-k_r\geq 0$, $ r\in \{p+1, \dots , 2p-1\}-\{q\}$. 

\noindent c2)  If $0<k_{2p}<\frac{k_p^2}{4}$ and $k_{p+1}= \dots  =k_{2p-1}=0$, then 

$P(x)=\left(x^p-\left(-\frac{k_p}{2}-\sqrt{\frac{k_p^2}{4}-k_{2p}}\right)\right)\left(x^p-\left(-\frac{k_p}{2}+\sqrt{\frac{k_p^2}{4}-k_{2p}}\right)\right)$ 

\noindent is $\mathcal{IR}$ by the digraph formed by the basic $2p$-cycle {\color{red}$C_{2p,1}\left(\frac{k_p^2}{4}-k_{2p}\right)$}, and the $p$-cycles $C_{p,1}\left(-\frac{k_p}{2}\right)$ and $C_{2p,p+1}\left(-\frac{k_p}{2}\right)$. Then {\color{red}$a_{2p,1}=\frac{k_p^2}{4}-k_{2p}>0$},
$a_{p,1}=a_{2p,p+1}=-\frac{k_p}{2}>0$.

\noindent c3)  If $0<k_{2p}=\frac{k_3^2}{4}$ and $k_{p+1}= \dots  =k_{2p+1}=0$, then 
$P(x)=\left(x^p+\frac{k_p}{2}\right)^2$ does not have simple Perron root so, is not $\mathcal{IR}$. This is the exceptional case.\end{proof}

\begin{teo} {\rm ($n=2p+1$)} \label{n=2p+1}
Let $P(x)=x^{2p+1}+k_px^{p+1}+k_{p+1}x^{p}+ \dots +k_{2p}x+k_{2p+1} \neq x^{2p+1}$ with $p\geq 3$. Then $\mathcal{R}$ and $\mathcal{IR}$ are equivalent, except if 

\noindent 1) $k_p<0, k_{p+1}<0, k_{p+2}= \dots =k_{2p}=0$ and $k_{2p+1}=k_pk_{p+1}$; or

\noindent 2) $k_p<0, k_{p+1}<0, k_{p+2}= \dots =k_{2p-1}=0,  0<k_{2p}<\frac{k_p^2}{4}$ and $k_{2p+1}=k_{p+1}\Big(\frac{k_p}{2}-\sqrt{\frac{k_p^2}{4}-k_{2p}}\Big)$; or

\noindent 3) $k_p<0, k_{p+1}= \dots = k_{2p-1}=k_{2p+1}=0$ and $k_{2p}=\frac{k_p^2}{4}$;

\noindent  in which cases $\mathcal{IR}$ is not possible.
\end{teo}

\begin{proof} From Theorem \ref{t39}, $P(x)$ is $\mathcal{R}$ if and only if  $k_p, \dots  k_{2p-1}\leq 0, \, k_{2p}\leq \frac{k_p^2}{4}$ and 

$k_{2p+1}\leq\left\{
\begin{array}{ll}
	k_pk_{p+1}& \;\; \mbox{if }\;\;k_{2p}\leq 0,\\
	\vspace*{-.3cm} & \\
	k_{p+1}\Big(\frac{k_p}{2}-\sqrt{\frac{k_p^2}{4}-k_{2p}}\Big)&\;\; \mbox{if }\;\; k_{2p}>0.\\
\end{array}\right.$

\noindent a1) If $k_{2p}< 0$, and $k_p<0$ or $k_{2p+1}<k_pk_{p+1}$, then $P(x)$ is $\mathcal{IR}$ by the digraph formed by the non-disjoint cycles 
{\color{red}$C_{2p,1}\left(-k_{2p}\right)$} and {\color{red}$C_{2p+1,p+2}\left(-k_p\right)$}, or 
the basic $(2p+1)$-cycle {\color{red}$C_{2p+1,1}\left(k_pk_{p+1}-k_{2p+1}\right)$}, 
and the $r$-cycles $C_{r,1}\left(-k_r\right), p+1\leq r \leq 2p-1$.
Then {\color{red}$a_{2p,1}=-k_{2p}>0$}, {\color{red}$a_{2p+1,p+2}=-k_p>0$} or
{\color{red}$a_{2p+1,1}=k_pk_{p+1}-k_{2p+1}> 0$}, $a_{r,1}=-k_r\geq 0, p+1\leq r \leq 2p-1$. 

\noindent a2) If $k_{2p}< 0$, $k_p=k_{2p+1}=0$ and there exists a coefficient $q\in \{p+1, \dots , 2p-1\}$ such that $k_q<0$, 
then $P(x)$ is $\mathcal{IR}$ by the digraph formed by the two non-disjoint cycles {\color{red}$C_{2p+1,2}\left(-k_{2p}\right)$} 
and {\color{red}$C_{q,1}\left(-k_q\right)$}, and the $r$-cycles $C_{r,1}\left(-k_r\right), \, r\in \{p+1, \dots , 2p-1\}-\{q\}$.
Then {\color{red}$a_{2p+1,2}=-k_{2p}>0$}, {\color{red}$a_{q,1}=-k_q>0$}, $a_{r,1}=-k_r\geq 0, \, r\in \{p+1, \dots , 2p-1\}-\{q\}$.

\noindent a3) If $k_{2p}< 0$ and $k_p= \dots =k_{2p-1}=k_{2p+1}=0$, then $P(x)=x(x^{2p}+k_{2p})$ so, by Theorem \ref{SFS}, is $\mathcal{IR}$.

\noindent b1) If $k_{2p}= 0$ and $k_{2p+1}<k_pk_{p+1}$, then $P(x)$ is $\mathcal{IR}$ by the digraph formed by the basic $(2p+1)$-cycle 
 {\color{red}$C_{2p+1,1}\left(k_pk_{p+1}-k_{2p+1}\right)$}, 
and the $r$-cycles $C_{r,1}\left(-k_r\right),  p+1\leq r \leq 2p-1$, and $C_{2p+1,p+2}\left(-k_p\right)$.
Then {\color{red}$a_{2p+1,1}=k_pk_{p+1}-k_{2p+1}>0$}, $a_{r,1}=-k_r\geq 0,  p+1\leq r \leq 2p-1, a_{2p+1,p+2}=-k_p\geq 0$.

\noindent b2) If $k_{2p}= 0$, $k_{2p+1}=k_pk_{p+1}$, $k_p<0$, $k_{p+1}<0$ and $k_{p+2}= \dots = k_{2p-1}=0$, then $P(x)=(x^p+k_p)(x^{p+1}+k_{p+1})$ does not have 
Frobenius structure, so is not $\mathcal{IR}$ by Theorem \ref{Teo1}. This is the exceptional case 1).

\noindent b3) If $k_{2p}= 0$, $k_{2p+1}=k_pk_{p+1}$, with $k_p<0$ or $k_{p+1}<0$, and there exists $q\in \{p+2,\dots , 2p-1\}$ such that $k_q<0$, then $P(x)$ is $\mathcal{IR}$ by the digraph formed by:

 the two non-disjoint cycles {\color{red}$C_{p,1}\left(-k_p\right)$} and {\color{red}$C_{2p+1,2p+2-q}\left(-k_{q}\right)$}, the cycle $C_{2p+1,p+1}\left(-k_{p+1}\right)$ and the $r$-cycles 
$C_{r,1}\left(-k_r\right), r\in \{p+2,\dots , 2p-1\}-\{q\}$, with {\color{red}$a_{p,1}=-k_p>0$}, {\color{red}$a_{2p+1,2p+2-q}=-k_{q}>0$}, $a_{2p+1,p+1}=-k_{p+1}\geq 0$,
$a_{r,1}=-k_r\geq 0, r\in \{p+2,\dots , 2p-1\}-\{q\}$, or 

the two non-disjoint cycles {\color{red}$C_{p+1,1}\left(-k_{p+1}\right)$} and {\color{red}$C_{2p+1,2p+2-q}\left(-k_{q}\right)$}, the cycle $C_{2p+1,p+2}\left(-k_{p}\right)$ and the $r$-cycles  $C_{r,1}\left(-k_r\right), r\in \{p+2,\dots , 2p-1\}-\{q\}$, with {\color{red}$a_{p+1,1}=-k_{p+1}>0$}, {\color{red}$a_{2p+1,2p+2-q}=-k_{q}>0$}, $a_{2p+1,p+2}=-k_{p}\geq 0$,  $a_{r,1}=-k_r\geq 0, r\in \{p+2,\dots , 2p-1\}-\{q\}$.

\noindent b4) If $k_{2p}= 0$, and only one of the coefficients $k_p, k_{p+1}, \dots , k_{2p-1}$ is negative, then the respective polynomials 
$x^{p+1}(x^p+k_p), x^p(x^{p+1}+k_{p+1}), \dots , x^2(x^{2p-1}+k_{2p-1})$ are $\mathcal{IR}$
by Theorem \ref{SFS}. 

\noindent c1)  If $0<k_{2p}\leq \frac{k_p^2}{4}$ and $k_{2p+1}<k_{p+1}\left(\frac{k_p}{2}-\sqrt{\frac{k_p^2}{4}-k_{2p}}\right)$, then $P(x)$ is $\mathcal{IR}$ 
by the digraph formed by the basic $(2p+1)$-cycle  {\color{red}$C_{2p+1,1}\left(k_{p+1}\left(\frac{k_p}{2}-\sqrt{\frac{k_p^2}{4}-k_{2p}}\right)-k_{2p+1}\right)$}, the $p$-cycles $C_{p,1}\left(-\frac{k_p}{2}-\sqrt{\frac{k_p^2}{4}-k_{2p}}\right)$ and  
$C_{2p+1,p+2}\left(-\frac{k_p}{2}+\sqrt{\frac{k_p^2}{4}-k_{2p}}\right)$, and the $r$-cycles $C_{r,1}\left(-k_r\right), p+1 \leq r \leq 2p-1$. Then {\color{red}$a_{2p+1,1} =k_{p+1}\left(\frac{k_p}{2}-\sqrt{\frac{k_p^2}{4}-k_{2p}}\right)-k_{2p+1}>0$}, 
$a_{p1} =-\frac{k_p}{2}-\sqrt{\frac{k_p^2}{4}-k_{2p}}> 0$, $a_{2p+1,p+2} =-\frac{k_p}{2}+\sqrt{\frac{k_p^2}{4}-k_{2p}}> 0$, $a_{r,1}=-k_r\geq 0, p+1 \leq r \leq 2p-1$.

\noindent c2)  If $0<k_{2p}\leq \frac{k_p^2}{4}$, there exists $q\in \{p+2, \dots , 2p-1\}$ with $k_q<0$, 
and $k_{2p+1}=k_{p+1}\left(\frac{k_p}{2}-\sqrt{\frac{k_p^2}{4}-k_{2p}}\right)$, then $P(x)$ is $\mathcal{IR}$ by the digraph formed by the two non-disjoint cycles {\color{red}$C_{q,1}(-k_q)$}
 and {\color{red}$C_{2p+1,p+2}\left(-\frac{k_p}{2}+\sqrt{\frac{k_p^2}{4}-k_{2p}}\right)$}, and the cycles $C_{p,1}\left(-\frac{k_p}{2}-\sqrt{\frac{k_p^2}{4}-k_{2p}}\right)$ and  $C_{r,1}(-k_r), r\in \{p+1,p+2, \dots , 2p-1\}-\{q\}$. Then {\color{red}$a_{q,1}=-k_q>0$}, {\color{red}$a_{2p+1,p+2}=-\frac{k_p}{2}+\sqrt{\frac{k_p^2}{4}-k_{2p}}>0$}, $a_{p,1}=-\frac{k_p}{2}-\sqrt{\frac{k_p^2}{4}-k_{2p}}>0$,  $a_{r,1}=-k_r\geq 0, r\in \{p+1,p+2, \dots , 2p-1\}-\{q\}$.

\noindent c3) Let $0<k_{2p}< \frac{k_p^2}{4}$, $k_{p+1}< 0$, $k_{p+2}= \dots =k_{2p-1}=0$ and $k_{2p+1}=k_{p+1}\left(\frac{k_p}{2}-\sqrt{\frac{k_p^2}{4}-k_{2p}}\right)$. 

\noindent Since $k_{2p}>0$, there exists $m_0\geq -\frac{k_p}{2}$ such that $k_{2p}=m_0(-k_p-m_0)$. Then $m_0=-\frac{k_p}{2}+\sqrt{\frac{k_p^2}{4}-k_{2p}}$.

\noindent This situation correponds to a digraph with two disjoint $p$-cycles with weights $m_0$ and $-k_p-m_0$ and without $2p$-cycles. 
And this is the optimum situation (see the proof of Theorem \ref{t39} in \cite{TAAMP}).

The value $k_{p+1}\left(\frac{k_p}{2}-\sqrt{\frac{k_p^2}{4}-k_{2p}}\right)$ is the maximum value of the coefficient $k_{2p+1}$ and so it is obtained with
the absence of $(2p+1)$-cycles and with a single $(p+1)$-cycle disjoint with the $p$-cycle of maximum weight $m_0$. 
This digraph with two disjoint $p$-cycles (by instance $C_{p,1}(-k_p-m_0)$ and $C_{2p+1,p+2}(m_0)$) and a $(p+1)$-cycle disjoint with $C_{2p+1,p+2}(m_0)$ 
(by instance $C_{p+1,1}(-k_{p+1})$) is not strongly connected and so there is no $\mathcal{IR}$.
This is the exceptional case 2).

\noindent c4) If $0<k_{2p}<\frac{k_p^2}{4}$ and $k_{p+1}= \dots =k_{2p-1}=k_{2p+1}=0$, then $P(x)$ is $\mathcal{IR}$ by the digraph formed by 
the two non-disjoint cycles {\color{red}$C_{2p,1}\left(\frac{k_p^2}{2}-k_{2p}\right)$} and {\color{red}$C_{2p+1,p+2}\left(-\frac{k_p}{2}\right)$}, 
and the cycle $C_{p,1}\left(-\frac{k_p}{2}\right)$. Then {\color{red}$a_{2p,1}=\frac{k_p^2}{2}-k_{2p}>0$}, {\color{red}$a_{2p+1,p+2}=-\frac{k_p}{2}>0$}, 
$a_{p,1}=-\frac{k_p}{2}>0$.

\noindent c5)  If $0<k_{2p}=\frac{k_p^2}{4}$, $k_{p+1}<0$, $k_{p+2}=\dots = k_{2p-1}=0$ and  $k_{2p+1}=\frac{k_pk_{p+1}}{2}$, then $P(x)$ is $\mathcal{IR}$ by the digraph formed by 
the two non-disjoint $(p+1)$-cycles {\color{red}$C_{p+1,1}\left(-\frac{k_{p+1}}{2}\right)$} and {\color{red}$C_{2p+1,p+1}\left(-\frac{k_{p+1}}{2}\right)$}, and the $p$-cycles $C_{p,1}\left(-\frac{k_p}{2}\right)$ and  $C_{2p+1,p+2}\left(-\frac{k_p}{2}\right)$.
Then {\color{red}$a_{p+1,1}=a_{2p+1,p+1}=-\frac{k_{p+1}}{2}>0$}, $a_{p,1}=a_{2p+1,p+2}=-\frac{k_p}{2}>0$.

\noindent c6)  If $0<k_{2p}=\frac{k_p^2}{4}$ and $k_{p+1}= \dots = k_{2p-1}=k_{2p+1}=0$, then $P(x)=x\left(x^p+\frac{k_p}{2}\right)^2$ does not have simple Perron root 
so, is not $\mathcal{IR}$. This is the exceptional case 3).\end{proof}

Since all the strongly connected realizing digraphs in the previous results are simple EBL digraphs, we have the following summary result.

\begin{co} \label{summary}
Let $P(x)=x^{n}+k_px^{n-p}+ \dots +k_{n-1}x+k_{n}$ with $2\leq p\leq n \leq 2p+1$. If  $P(x)$ is $\mathcal{IR}$, 
then it is $\mathcal{IR}$ by a simple EBL digraph.
\end{co} 

In an Appendix, we give a very explicit description of our trace 0 realizations in case $n=6$ or 7, as examples for comparison and clarity.

\section*{Appendix}

\begin{teo} {\rm ($n=6$)} \label{n=6}
Let $P(x)=x^6+k_3x^3+k_4x^2+k_5x+k_6 \neq x^6$. Then $\mathcal{R}$ and $\mathcal{IR}$ are equivalent, except if $P(x)=x^6+k_3x^3+(\frac{k_3}{2})^2$ with $k_3<0$, in which case no $\mathcal{IR}$ is possible.
\end{teo}

\begin{proof} 
$P(x)$ is $\mathcal{R}$ if and only if  $k_3\leq 0, \, k_4\leq 0, \, k_5\leq 0$ and $k_6\leq \frac{k_3^2}{4}$.

\noindent
a) If $k_6<0$, then $P(x)$ is $\mathcal{IR}$ by the digraph formed by the basic 6-cycle {\color{red}$C_{6,1}\left(-k_6\right)$} and the cycles 
$C_{3,1}\left(-k_3\right), \, C_{4,1}\left(-k_4\right)$ and  $C_{5,1}\left(-k_5\right)$ with matrix  
$\left[\begin{array}{cccccc}
0 & {\color{red}1} & 0 & 0 & 0 & 0\\
0 & 0 & {\color{red}1} & 0 & 0 & 0\\
-k_3 & 0 & 0 & {\color{red}1} & 0 & 0\\
-k_4 & 0 & 0 &  0 & {\color{red}1} & 0\\
-k_5 & 0 & 0 & 0 & 0 & {\color{red}1}\\
{\color{red}-k_6} & 0 & 0 & 0 & 0 & 0
\end{array}\right]$.

\noindent b1) If $k_6=0$, and at least two of the coefficients $k_3, k_4, k_5$ are negative, then $P(x)$ is $\mathcal{IR}$. For instance, if $k_3,k_4<0$
we build the digraph with the two non-disjoint cycles {\color{red}$C_{3,1}\left(-k_3\right)$} and {\color{red}$C_{6,3}\left(-k_4\right)$}, and the cycle 
$C_{5,1}\left(-k_5\right)$. Analogously in the other cases. Thus, we obtain the matrices 

\noindent $\left[\begin{array}{cccccc}
0 & {\color{red}1} & 0 & 0 & 0 & 0\\
0 & 0 & {\color{red}1} & 0 & 0 & 0\\
{\color{red}-k_3} & 0 & 0 & {\color{red}1} & 0 & 0\\
0 & 0 & 0 &  0 & {\color{red}1} & 0\\
-k_5 & 0 & 0 & 0 & 0 & {\color{red}1}\\
0 & 0 & {\color{red}-k_4} & 0 & 0 & 0
\end{array}\right]$, \, 
$\left[\begin{array}{cccccc}
0 & {\color{red}1} & 0 & 0 & 0 & 0\\
0 & 0 & {\color{red}1} & 0 & 0 & 0\\
{\color{red}-k_3} & 0 & 0 & {\color{red}1} & 0 & 0\\
-k_4 & 0 & 0 &  0 & {\color{red}1} & 0\\
0 & 0 & 0 & 0 & 0 & {\color{red}1}\\
0 & {\color{red}-k_5} & 0 & 0 & 0 & 0
\end{array}\right]$, \,
$\left[\begin{array}{cccccc}
0 & {\color{red}1} & 0 & 0 & 0 & 0\\
0 & 0 & {\color{red}1} & 0 & 0 & 0\\
-k_3 & 0 & 0 & {\color{red}1} & 0 & 0\\
{\color{red}-k_4} & 0 & 0 &  0 & {\color{red}1} & 0\\
0 & 0 & 0 & 0 & 0 & {\color{red}1}\\
0 & {\color{red}-k_5} & 0 & 0 & 0 & 0
\end{array}\right]$.

\noindent b2) If $k_6=0$, and just one of the coefficients $k_3, k_4, k_5$ is negative, then the corresponding polynomials 
$x^3\left(x^3+k_3\right), x^2\left(x^4+k_4\right),x\left(x^5+k_5\right)$ are $\mathcal{IR}$ by Theorem \ref{SFS} by the respective matrices 

\noindent $\left[\begin{array}{cccccc}
0 & {\color{red}1} & 0 & 0 & 0 & 0\\
0 & 0 & {\color{red}1} & 0 & 0 & 0\\
-\frac{k_3}{2} & 0 & 0 & {\color{red}1} & 0 & 0\\
0 & 0 & 0 &  0 & {\color{red}1} & 0\\
0 & 0 & 0 & 0 & 0 & {\color{red}1}\\
{\color{red}\frac{k_3^2}{4}} & 0 & 0 & -\frac{k_3}{2} & 0 & 0
\end{array}\right]$, \, 
$\left[\begin{array}{cccccc}
0 & {\color{red}1} & 0 & 0 & 0 & 0\\
0 & 0 & {\color{red}1} & 0 & 0 & 0\\
0 & 0 & 0 & {\color{red}1} & 0 & 0\\
{\color{red}-\frac{k_4}{2}}  & 0 & 0 &  0 & {\color{red}1} & 0\\
0 & 0 & 0 & 0 & 0 & {\color{red}1}\\
0 & 0 & {\color{red}-\frac{k_4}{2}} & 0 & 0 & 0
\end{array}\right]$, \,
$\left[\begin{array}{cccccc}
0 & {\color{red}1} & 0 & 0 & 0 & 0\\
0 & 0 & {\color{red}1} & 0 & 0 & 0\\
0 & 0 & 0 & {\color{red}1} & 0 & 0\\
0 & 0 & 0 &  0 & {\color{red}1} & 0\\
{\color{red}-\frac{k_5}{2}}  & 0 & 0 & 0 & 0 & {\color{red}1}\\
0 & {\color{red}-\frac{k_5}{2}}  & 0 & 0 & 0 & 0
\end{array}\right]$.

\noindent c1) If $0<k_6\leq \frac{k_3^2}{4}$, and at least one of the coefficients $k_4, k_5$ is negative, then $P(x)$ is $\mathcal{IR}$. For instance, if $k_4<0$
we build the digraph with two non-disjoint cycles {\color{red}$C_{4,1}\left(-k_4\right)$} and {\color{red}$C_{6,4}\left(a_{64}\right)$}, and the cycles 
$C_{3,1}\left(a_{31}\right)$ and $C_{5,1}\left(-k_5\right)$. Analogously in the other case. Thus, we obtain the matrices 
\noindent $\left[\begin{array}{cccccc}
0 & {\color{red}1} & 0 & 0 & 0 & 0\\
0 & 0 & {\color{red}1} & 0 & 0 & 0\\
a_{31} & 0 & 0 & {\color{red}1} & 0 & 0\\
{\color{red}-k_4} & 0 & 0 &  0 & {\color{red}1} & 0\\
-k_5 & 0 & 0 & 0 & 0 & {\color{red}1}\\
0 & 0 & 0 & {\color{red}a_{64}} & 0 & 0
\end{array}\right]$, \, 
$\left[\begin{array}{cccccc}
0 & {\color{red}1} & 0 & 0 & 0 & 0\\
0 & 0 & {\color{red}1} & 0 & 0 & 0\\
a_{31} & 0 & 0 & {\color{red}1} & 0 & 0\\
-k_4 & 0 & 0 &  0 & {\color{red}1} & 0\\
{\color{red}-k_5} & 0 & 0 & 0 & 0 & {\color{red}1}\\
0 & 0 & 0 & {\color{red}a_{64}} & 0 & 0
\end{array}\right]$ \hspace{0.1cm}{where}\hspace{0.1cm}
${\left\{
\begin{array}{rl}
a_{31} & =-\frac{k_3}{2}-\sqrt{\frac{k_3^2}{4}-k_6},  \\
{\color{red}a_{64}} & ={\color{red}-\frac{k_3}{2}+\sqrt{\frac{k_3^2}{4}-k_6}}.
\end{array}
\right.}$

\noindent c2)  If $0<k_6<\frac{k_3^2}{4}$ and $k_4=k_5=0$, then 
$P(x)=\left(x^3-\left(-\frac{k_3}{2}-\sqrt{\frac{k_3^2}{4}-k_6}\right)\right)\left(x^3-\left(-\frac{k_3}{2}+\sqrt{\frac{k_3^2}{4}-k_6}\right)\right)$ 
 is $\mathcal{IR}$ by the digraph formed by the basic $6$-cycle
{\color{red}$C_{6,1}\left(\frac{k_3^2}{4}-k_6\right)$}, and the $3$-cycles $C_{3,1}\left(-\frac{k_3}{2}\right)$ and $C_{6,4}\left(-\frac{k_3}{2}\right)$, 
with matrix
$\left[\begin{array}{cccccc}
0 & {\color{red}1} & 0 & 0 & 0 & 0\\
0 & 0 & {\color{red}1} & 0 & 0 & 0\\
-\frac{k_3}{2} & 0 & 0 & {\color{red}1} & 0 & 0\\
0 & 0 & 0 &  0 & {\color{red}1} & 0\\
0 & 0 & 0 & 0 & 0 & {\color{red}1}\\
{\color{red}\frac{k_3^2}{4}-k_6} & 0 & 0 & -\frac{k_3}{2} & 0 & 0
\end{array}\right]$.

\noindent c3)  If $0<k_6=\frac{k_3^2}{4}$ and $k_4=k_5=0$, then 
$P(x)=\left(x^3+\frac{k_3}{2}\right)^2$ does not have simple Perron root so is not $\mathcal{IR}$. This is the exceptional case.\end{proof}

\begin{teo} {\rm ($n=7$)} \label{n=7}
Let $P(x)=x^7+k_3x^4+k_4x^3+k_5x^2+k_6x+k_7 \neq x^7$. Then $\mathcal{R}$ and $\mathcal{IR}$ are equivalent, except if 

\noindent 1) $k_3<0, k_4<0, k_5=k_6=0$ and $k_7=k_3k_4$; or

\noindent 2) $k_3<0, k_4<0, k_5=0, 0<k_6<\frac{k_3^2}{4}$ and $k_7=k_{4}\Big(\frac{k_3}{2}-\sqrt{\frac{k_3^2}{4}-k_{6}}\Big)$; or

\noindent 3) $k_3<0, k_4=k_5=0, k_6=\frac{k_3^2}{4}$ and $k_7=0$;

\noindent in which cases no $\mathcal{IR}$ representation is possible.

\end{teo}

\begin{proof} 
$P(x)$ is $\mathcal{R}$ if and only if  $k_3\leq 0, \, k_4\leq 0, \, k_5\leq 0, \, k_6\leq \frac{k_3^2}{4}$ and 
$k_{7}\leq\left\{
\begin{array}{ll}
	k_3k_{4}& \;\; \mbox{if }\;\;k_{6}\leq 0,\\
	\vspace*{-.3cm} & \\
	k_{4}\Big(\frac{k_3}{2}-\sqrt{\frac{k_3^2}{4}-k_{6}}\Big)&\;\; \mbox{if }\;\; k_{6}>0.\\
\end{array}\right.$

\noindent a1) If $k_6<0$, and $k_3<0$ or $k_7<k_3k_4$, then $P(x)$ is $\mathcal{IR}$ by the digraph formed by the non-disjoint cycles 
{\color{red}$C_{6,1}\left(-k_6\right)$}, and {\color{red}$C_{7,5}\left(-k_3\right)$} or 
the basic $7$-cycle {\color{red}$C_{7,1}\left(k_3k_4-k_7\right)$}, 
and the cycles $C_{4,1}\left(-k_4\right)$ and $C_{5,1}\left(-k_5\right)$,  with matrix 
$\left[\begin{array}{ccccccc}
0 & {\color{red}1} & 0 & 0 & 0 & 0 & 0\\
0 & 0 & {\color{red}1} & 0 & 0 & 0 & 0\\
0 & 0 & 0 & {\color{red}1} & 0 & 0 & 0\\
-k_4 & 0 & 0 & 0 & {\color{red}1} & 0 & 0\\
-k_5 & 0 & 0 & 0 & 0 & {\color{red}1} & 0\\
{\color{red}-k_6} & 0 & 0 & 0 & 0 & 0 & {\color{red}1}\\
{\color{red}k_3k_4-k_7} & 0 & 0 & 0 & {\color{red}-k_3} & 0 & 0\\
\end{array}\right]$.

\noindent a2) If $k_6<0$,  $k_3=k_7=0$, and at least one of the coefficients $k_4, k_5$ is negative, then $P(x)$ is $\mathcal{IR}$ by the digraph formed by  
the non-disjoint cycles {\color{red}$C_{7,2}\left(-k_6\right)$}, and {\color{red}$C_{4,1}\left(-k_4\right)$} or {\color{red}$C_{5,1}\left(-k_5\right)$}, with matrices 
$\left[\begin{array}{ccccccc}
0 & {\color{red}1} & 0 & 0 & 0 & 0 & 0\\
0 & 0 & {\color{red}1} & 0 & 0 & 0 & 0\\
0 & 0 & 0 & {\color{red}1} & 0 & 0 & 0\\
{\color{red}-k_4} & 0 & 0 & 0 & {\color{red}1} & 0 & 0\\
-k_5 & 0 & 0 & 0 & 0 & {\color{red}1} & 0\\
0 & 0 & 0 & 0 & 0 & 0 & {\color{red}1}\\
0 & {\color{red}-k_6} & 0 & 0 & 0 & 0 & 0\\
\end{array}\right]$ or 
$\left[\begin{array}{ccccccc}
0 & {\color{red}1} & 0 & 0 & 0 & 0 & 0\\
0 & 0 & {\color{red}1} & 0 & 0 & 0 & 0\\
0 & 0 & 0 & {\color{red}1} & 0 & 0 & 0\\
-k_4 & 0 & 0 & 0 & {\color{red}1} & 0 & 0\\
{\color{red}-k_5} & 0 & 0 & 0 & 0 & {\color{red}1} & 0\\
0 & 0 & 0 & 0 & 0 & 0 & {\color{red}1}\\
0 & {\color{red}-k_6} & 0 & 0 & 0 & 0 & 0\\
\end{array}\right]$.

\noindent a3) If $k_6<0$, and $k_3=k_4=k_5=k_7=0$, then $P(x)=x(x^6+k_6)$, by Theorem \ref{SFS}, is $\mathcal{IR}$
by the matrix
$\left[\begin{array}{ccccccc}
0 & {\color{red}1} & 0 & 0 & 0 & 0 & 0\\
0 & 0 & {\color{red}1} & 0 & 0 & 0 & 0\\
0 & 0 & 0 & {\color{red}1} & 0 & 0 & 0\\
0 & 0 & 0 & 0 & {\color{red}1} & 0 & 0\\
0 & 0 & 0 & 0 & 0 & {\color{red}1} & 0\\
{\color{red}-\frac{k_6}{2}} & 0 & 0 & 0 & 0 & 0 & {\color{red}1}\\
0 & {\color{red}-\frac{k_6}{2}} & 0 & 0 & 0 & 0 & 0\\
\end{array}\right]$.

\noindent b1) If $k_6=0$ and $k_7<k_3k_4$, then $P(x)$ is $\mathcal{IR}$ by the digraph formed by the basic $7$-cycle {\color{red}$C_{7,1}\left(k_3k_4-k_7\right)$}, 
and the cycles $C_{4,1}\left(-k_4\right)$, $C_{5,1}\left(-k_5\right)$ and $C_{7,5}\left(-k_3\right)$,  with matrix  

$\left[\begin{array}{ccccccc}
0 & {\color{red}1} & 0 & 0 & 0 & 0 & 0\\
0 & 0 & {\color{red}1} & 0 & 0 & 0 & 0\\
0 & 0 & 0 & {\color{red}1} & 0 & 0 & 0\\
-k_4 & 0 & 0 & 0 & {\color{red}1} & 0 & 0\\
-k_5 & 0 & 0 & 0 & 0 & {\color{red}1} & 0\\
0 & 0 & 0 & 0 & 0 & 0 & {\color{red}1}\\
{\color{red}k_3k_4-k_7} & 0 & 0 & 0 & -k_3 & 0 & 0\\
\end{array}\right]$.

\noindent b2) If $k_6=0$, $k_7=k_3k_4$, $k_3,k_4<0$ and $k_5=0$, then $P(x)=(x^3+k_3)(x^4+k_4)$ does not have Frobenius structure so, by Theorem \ref{Teo1}, is not $\mathcal{IR}$. This is the exceptional case 1).

\noindent b3) If $k_6=0$, $k_7=k_3k_4$, $k_3<0$ or $k_4<0$, and $k_5<0$, then $P(x)$ is $\mathcal{IR}$ by the digraph formed by:
the two non-disjoint cycles {\color{red}$C_{3,1}\left(-k_3\right)$} and {\color{red}$C_{7,3}\left(-k_5\right)$}, and the $4$-cycle $C_{7,4}\left(-k_4\right)$; or 
the two non-disjoint cycles {\color{red}$C_{4,1}\left(-k_4\right)$} and {\color{red}$C_{7,3}\left(-k_5\right)$}, and the $3$-cycle $C_{7,5}\left(-k_3\right)$, with respective matrices
$\left[\begin{array}{ccccccc}
0 & {\color{red}1} & 0 & 0 & 0 & 0 & 0\\
0 & 0 & {\color{red}1} & 0 & 0 & 0 & 0\\
{\color{red}-k_3} & 0 & 0 & {\color{red}1} & 0 & 0 & 0\\
0& 0 & 0 & 0 & {\color{red}1} & 0 & 0\\
0 & 0 & 0 & 0 & 0 & {\color{red}1} & 0\\
0 & 0 & 0 & 0 & 0 & 0 & {\color{red}1}\\
0& 0 & {\color{red}-k_5} & -k_4  & 0 & 0 & 0\\
\end{array}\right]$ \; or \;
$\left[\begin{array}{ccccccc}
0 & {\color{red}1} & 0 & 0 & 0 & 0 & 0\\
0 & 0 & {\color{red}1} & 0 & 0 & 0 & 0\\
0 & 0 & 0 & {\color{red}1} & 0 & 0 & 0\\
{\color{red}-k_4} & 0 & 0 & 0 & {\color{red}1} & 0 & 0\\
0 & 0 & 0 & 0 & 0 & {\color{red}1} & 0\\
0 & 0 & 0 & 0 & 0 & 0 & {\color{red}1}\\
0 & 0 & {\color{red}-k_5}  & 0 & -k_3 & 0 & 0\\
\end{array}\right]$.

\noindent b4) If $k_6=0$, and only one of the coefficients $k_3, k_4, k_5 $ is negative, then the respective polynomials are
$x^4(x^3+k_3), x^3(x^4+k_4), x^2(x^5+k_5)$ and by Theorem \ref{SFS} they are $\mathcal{IR}$ by the respective matrices

\small
\noindent 
$\left[\begin{array}{ccccccc}
0 & {\color{red}1} & 0 & 0 & 0 & 0 & 0\\
0 & 0 & {\color{red}1} & 0 & 0 & 0 & 0\\
-\frac{k_3}{2} & 0 & 0 & {\color{red}1} & 0 & 0 & 0\\
0 & 0 & 0 & 0 & {\color{red}1} & 0 & 0\\
0 & 0 & 0 & 0 & 0 & {\color{red}1} & 0\\
{\color{red}\frac{k_3^2}{4}} & 0 & 0 & 0 & 0 & 0 & {\color{red}1}\\
0 & 0 & 0 & 0 & {\color{red}-\frac{k_3}{2}} & 0 & 0\\
\end{array}\right]$,  
$\left[\begin{array}{ccccccc}
0 & {\color{red}1} & 0 & 0 & 0 & 0 & 0\\
0 & 0 & {\color{red}1} & 0 & 0 & 0 & 0\\
0 & 0 & 0 & {\color{red}1} & 0 & 0 & 0\\
{\color{red}-\frac{k_4}{2}} & 0 & 0 & 0 & {\color{red}1} & 0 & 0\\
0 & 0 & 0 & 0 & 0 & {\color{red}1} & 0\\
0 & 0 & 0 & 0 & 0 & 0 & {\color{red}1}\\
0 & 0 & 0 & {\color{red}-\frac{k_4}{2}} & 0 & 0 & 0\\
\end{array}\right]$,  
$\left[\begin{array}{ccccccc}
0 & {\color{red}1} & 0 & 0 & 0 & 0 & 0\\
0 & 0 & {\color{red}1} & 0 & 0 & 0 & 0\\
0 & 0 & 0 & {\color{red}1} & 0 & 0 & 0\\
0 & 0 & 0 & 0 & {\color{red}1} & 0 & 0\\
{\color{red}-\frac{k_5}{2}} & 0 & 0 & 0 & 0 & {\color{red}1} & 0\\
0 & 0 & 0 & 0 & 0 & 0 & {\color{red}1}\\
0 & 0 & {\color{red}-\frac{k_5}{2}}  & 0 & 0 & 0 & 0\\
\end{array}\right]$.

\normalsize
\noindent c1)  If $0<k_6\leq \frac{k_3^2}{4}$ and $k_7<k_4\left(\frac{k_3}{2}-\sqrt{\frac{k_3^2}{4}-k_6}\right)$, then $P(x)$ is $\mathcal{IR}$ by the digraph formed by
the basic $7$-cycle  {\color{red}$C_{7,1}\left(a_{71}\right)$}, the $3$-cycles $C_{3,1}\left(a_{31}\right)$ and $C_{7,5}\left(a_{75}\right)$, and the cycles  $C_{4,1}\left(-k_4\right)$, and $C_{5,1}\left(-k_5\right)$, with matrix
 
$\left[\begin{array}{ccccccc}
0 & {\color{red}1} & 0 & 0 & 0 & 0 & 0\\
0 & 0 & {\color{red}1} & 0 & 0 & 0 & 0\\
a_{31} & 0 & 0 & {\color{red}1} & 0 & 0 & 0\\
-k_4 & 0 & 0 & 0 & {\color{red}1} & 0 & 0\\
-k_5 & 0 & 0 & 0 & 0 & {\color{red}1} & 0\\
0 & 0 & 0 & 0 & 0 & 0 & {\color{red}1}\\
{\color{red}a_{71}} & 0 & 0  & 0 & a_{75} & 0 & 0\\
\end{array}\right]$ \hspace{1em}{where}\hspace{1em}
${\left\{
\begin{array}{rl}
a_{31} & =-\frac{k_3}{2}-\sqrt{\frac{k_3^2}{4}-k_4},  \\
a_{75} & =-\frac{k_3}{2}+\sqrt{\frac{k_3^2}{4}-k_4},\\
{\color{red}a_{71}} & ={\color{red} k_4\left(\frac{k_3}{2}-\sqrt{\frac{k_3^2}{4}-k_6}\right)-k_7}>0 .\\
\end{array}
\right.}$

\noindent c2) If $0<k_{6}\leq \frac{k_3^2}{4}$, $k_{5}<0$ and $k_{7}=k_{4}\left(\frac{k_3}{2}-\sqrt{\frac{k_3^2}{4}-k_{6}}\right)$, then
$P(x)$ is $\mathcal{IR}$ by the digraph formed by the two non-disjoint cycles 
{\color{red}$C_{5,1}\left(-k_5\right)$} and {\color{red}$C_{7,5}\left(-\frac{k_3}{2}+\sqrt{\frac{k_3^2}{4}-k_6}\right)$}, 
and the cycles $C_{3,1}\left(-\frac{k_3}{2}-\sqrt{\frac{k_3^2}{4}-k_6}\right)$ and $C_{4,1}\left(-k_4\right)$ with matrix

$\left[\begin{array}{ccccccc}
0 & {\color{red}1} & 0 & 0 & 0 & 0 & 0\\
0 & 0 & {\color{red}1} & 0 & 0 & 0 & 0\\
-\frac{k_3}{2}-\sqrt{\frac{k_3^2}{4}-k_6}& 0 & 0 & {\color{red}1} & 0 & 0 & 0\\
-k_4 & 0 & 0 & 0 & {\color{red}1} & 0 & 0\\
{\color{red}-k_5} & 0 & 0 & 0 & 0 & {\color{red}1} & 0\\
0 & 0 & 0 & 0 & 0 & 0 & {\color{red}1}\\
0 & 0 &0  & 0 & {\color{red}-\frac{k_3}{2}+\sqrt{\frac{k_3^2}{4}-k_6}} & 0 & 0\\
\end{array}\right]$.

\noindent c3) Let $0<k_{6}< \frac{k_3^2}{4}$, $k_{4}< 0$, $k_{5}=0$ and $k_{7}=k_{4}\left(\frac{k_3}{2}-\sqrt{\frac{k_3^2}{4}-k_{6}}\right)$.

\noindent Since $k_{6}>0$, there exists $m_0\geq -\frac{k_3}{2}$ such that $k_{6}=m_0(-k_3-m_0)$. Then $m_0=-\frac{k_3}{2}+\sqrt{\frac{k_3^2}{4}-k_{6}}$.

\noindent This situation corresponds to a digraph with two disjoint $3$-cycles with weights $m_0$ and $-k_3-m_0$ and without $6$-cycles. 
And this is the optimum situation (see the proof of Theorem \ref{t39} in \cite{TAAMP}).

The value $k_{4}\left(\frac{k_3}{2}-\sqrt{\frac{k_3^2}{4}-k_{6}}\right)$ is the maximum value of the coefficient $k_{7}$ and so it is obtained with
the absence of $7$-cycles and with a single $4$-cycle (if it exists) disjoint with the $3$-cycle of maximum weight $m_0$. 
This digraph with two disjoint $3$-cycles (by instance $C_{31}(-k_3-m_0)$ and $C_{7,5}(-m_0)$) and a $4$-cycle disjoint with $C_{7,5}(-m_0)$ 
(by instance $C_{4,1}(-k_4)$) is not strongly connected and so there is no $\mathcal{IR}$.
This is the exceptional case 2).

\noindent c4)  If $0<k_6<\frac{k_3^2}{4}$, and $k_4=k_7=0$, then $P(x)$ is $\mathcal{IR}$ by the digraph formed by 
the two non-disjoint cycles {\color{red}$C_{6,1}\left(\frac{k_3^2}{2}-k_6\right)$} and {\color{red}$C_{7,5}\left(-\frac{k_3}{2}\right)$}, 
and the cycles $C_{3,1}\left(-\frac{k_3}{2}\right)$ and $C_{5,1}\left(-k_5\right)$ with matrix 
$\left[\begin{array}{ccccccc}
0 & {\color{red}1} & 0 & 0 & 0 & 0 & 0\\
0 & 0 & {\color{red}1} & 0 & 0 & 0 & 0\\
-\frac{k_3}{2} & 0 & 0 & {\color{red}1} & 0 & 0 & 0\\
0 & 0 & 0 & 0 & {\color{red}1} & 0 & 0\\
-k_5 & 0 & 0 & 0 & 0 & {\color{red}1} & 0\\
{\color{red}\frac{k_3^2}{2}-k_6} & 0 & 0 & 0 & 0 & 0 & {\color{red}1}\\
0 & 0 & 0  & 0 & {\color{red}-\frac{k_3}{2}}  & 0 & 0\\
\end{array}\right].$

\noindent c5) If $0<k_6= \frac{k_3^2}{4}$, $k_4<0$, $k_5=0$ and $k_7=\frac{k_3k_4}{2}$, then $P(x)$ is $\mathcal{IR}$ 
by the digraph formed by the two non-disjoint 4-cycles {\color{red}$C_{4,1}\left(-\frac{k_4}{2}\right)$} and {\color{red}$C_{7,4}\left(-\frac{k_4}{2}\right)$}, and the $3$-cycles $C_{3,1}\left(-\frac{k_3}{2}\right)$ and $C_{7,5}\left(-\frac{k_3}{2}\right)$ with matrix
$\left[\begin{array}{ccccccc}
0 & {\color{red}1} & 0 & 0 & 0 & 0 & 0\\
0 & 0 & {\color{red}1} & 0 & 0 & 0 & 0\\
-\frac{k_3}{2} & 0 & 0 & {\color{red}1} & 0 & 0 & 0\\
{\color{red}-\frac{k_4}{2}} & 0 & 0 & 0 & {\color{red}1} & 0 & 0\\
0 & 0 & 0 & 0 & 0 & {\color{red}1} & 0\\
0 & 0 & 0 & 0 & 0 & 0 & {\color{red}1}\\
0 & 0 &0  & {\color{red}-\frac{k_4}{2}} &  -\frac{k_3}{2} & 0 & 0\\
\end{array}\right]$.

\noindent c6)  If $0<k_6=\frac{k_3^2}{4}$ and $k_4=k_5=k_7=0$, then $P(x)=x\left(x^3+\frac{k_3}{2}\right)^2$ does not have simple Perron root so
 is not $\mathcal{IR}$. This is the exceptional case 3).\end{proof}


\begin{thebibliography}{99}


\bibitem{Brauer} 
A. Brauer, Limits for the characteristic roots of a matrix, IV: Applications to stochastic matrices,
\textit{Duke Math. J.} 19 (1952) 75-91.

\bibitem{CDS}
D. M. Cvetkovi\'c, M. Doob, H. Sachs,
Spectra of Graphs, Johann Ambrosius Barth, Heidelberg 1995.

\bibitem{Fi} 
M. Fiedler, Eigenvalues of nonnegative symmetric matrices,
\textit{Linear Algebra Appl.} 9 (1974) 119-142.

\bibitem{F}
G. Frobenius,
$\ddot{\rm U}$ber matrizen aus nicht negativen elementen,
\emph{S. -B. K. Preuss. Akad. Wiss. Berlin}, 1912 p. 456-477.

\bibitem{G}
F. R. Gantmacher, 
Th\`eorie des matrices, Tome 2, Collection Universitaire de Math\`ematiques, Dunod, Paris, 1966. 

\bibitem{HJ1}
R. Horn, C. R. Johnson,
Matrix Analysis, Cambridge University Press, 2013.


\bibitem{J} 
C. R. Johnson, Row stochastic matrices similar to doubly
stochastic matrices, \textit{Linear Multilinear Algebra} 10 (1981) 113-130.

\bibitem{JMPP}
C. R. Johnson, C. Mariju\'an, P. Papparella, M. Pisonero, The NIEP. Operator theory, operator algebras, and matrix theory, 199-220, Oper. Theory Adv. Appl. 267, Birkh$\ddot{\rm a}$user/Springer (2018), 353-372. ISBN: 978-3-319-72448-5. 

\bibitem{JP} 
C. R. Johnson, P. Paparella, Row cones, perron similarities, and nonnegative spectra, \textit{Linear Multilinear Algebra} 65 (2017) 2124-2130.


\bibitem{JMPS} A. I. Julio, C. Mariju\'an, M. Pisonero, R. L. Soto, On universal realizability of spectra, \textit{Linear Algebra Appl.} 563 (2019) 353-372.

\bibitem{JMPS1} A. I. Julio, C. Mariju\'an, M. Pisonero, R. L. Soto, Universal realizability in low dimension, \textit{Linear Algebra Appl.} 619 (2021) 107-136.

\bibitem{LM2} 
T. J. Laffey, E.\ Meehan, A characterization of trace zero nonnegative $5\times 5$ matrices, 
\textit{Linear Algebra Appl.} 302/303 (1999) 295-302.

\bibitem{LoLo} 
R. Loewy, D. London, A note on the inverse problem for
nonnegative matrices, \textit{Linear and Multilinear Algebra} 6 (1978) 83-90.

\bibitem{Mi1}
 H. Minc,
 Inverse elememtary divisor problem for nonnegative matrices
\emph{Proceedings of the American Mathematical Society}, 83 n4(1981) 665-669.


\bibitem{Pa} 
P. Paparella,  Realizing Suleimanova-type spectra via permutative matrices, \textit{Electron. J. Linear Algebra} 31 (2016) 306-312.

\bibitem{Pe1}
H. Perfect,
On positive stochastic matrices with real characteristic roots,
\emph{Proceedings of the Cambridge Philosophical Society},  48 (1952) 271-276.

\bibitem{Pe3} 
H. Perfect, Methods of constructing certain stochastic
matrices II, \textit{Duke Math.J.} 22 (1955) 305-311.

\bibitem{PeMi} 
H. Perfect, L. Mirsky, Spectral properties of doubly
stochastic matrices, \textit{Monatsh. Math.} 69 (1965) 35-57.


\bibitem{Sp}
O. Spector, A characterization of trace zero symmetric nonnegative 5$\times$5 matrices,
\textit{Linear Algebra Appl.} 434 (2011), n 4, 1000-1017.

\bibitem{Su} 
H. R. Sule\v{\i}manova, Stochastic matrices with real characteristic values,
\textit{ Dokl. Akad. Nauk. S.S.S.R.} 66 (1949)  343-345 (in Russian).

\bibitem{TAAMP}
J. Torre-Mayo, M. R. Abril-Raymundo, E. Alarcia-Est\'evez, C. Mariju\'an, M. Pisonero, The nonnegative inverse eigenvalue problem from the coefficients of the characteristic polynomial. EBL digraphs, {\it Linear Algebra Appl.} 426 (2007) 729-773.


\end{thebibliography}
\end{document}